\documentclass[a4,10pt,leqno]{amsart}
\usepackage{amsmath}
\usepackage{amsthm}
\usepackage{amsfonts}
\usepackage{mathtools}
\usepackage{amssymb}
\usepackage{graphicx}
\usepackage{epstopdf}
\usepackage{amscd}
\usepackage{color}
\usepackage{enumerate}
\usepackage{hyperref}
\theoremstyle{plain}
\newtheorem{theorem}{Theorem}[section]
\newtheorem{defn}{Definition}[section]
\newtheorem{lem}{Lemma}[section]
\newtheorem{prop}{Proposition}[section]
\newtheorem{rmk}{Remark}[section]
\newtheorem{cor}{Corollary}[section]
\DeclareMathOperator{\supp}{supp}
\def\C{\mathbf {C}}
\newcommand{\B}[1]{\mathbb{#1}}

\newcommand{\eps}{\varepsilon}

\makeatletter
\def\cleardoublepage{\clearpage\if@twoside \ifodd\c@page\else
\hbox{}
\thispagestyle{empty}
\newpage
\if@twocolumn\hbox{}\newpage\fi\fi\fi}
\makeatother
\title{Asymptotic behaviours in Fractional Orlicz-Sobolev spaces on Carnot groups}

\author{M. Capolli}
\address{Marco Capolli: Dipartimento di Matematica\\Universit\`a di Trento\\ Via Sommarive 14\\ 38123, Povo (Trento) - Italy\\}
\email{marco.capolli@unitn.it}
\thanks{M.C. is supported by MIUR, Italy, GNAMPA of INDAM and University of Trento, Italy.}
\author{A. Maione}
\address{Alberto Maione: Dipartimento di Matematica\\Universit\`a di Trento\\ Via Sommarive 14\\ 38123, Povo (Trento) - Italy\\}
\email{alberto.maione@unitn.it}
\thanks{A.M. is supported by MIUR, Italy, GNAMPA of INDAM and University of Trento, Italy.}
\author{A. M. Salort}
\address{Ariel M. Salort: Departamento de Matem\'atica, FCEyN - Universidad de Buenos Aires and IMAS - CONICET Ciudad Universitaria, Pabell\'on I (1428) Av. Cantilo s/n., Buenos Aires, Argentina.}
\email{asalort@dm.uba.ar}
\urladdr{http://mate.dm.uba.ar/~asalort}
\author{E. Vecchi}
\address{Eugenio Vecchi: Dipartimento di Matematica\\Politecnico di Milano\\ Piazza Leonardo da Vinci\\ 20133, Milano (MI) - Italy\\}
\email{eugenio.vecchi@polimi.it}

\date{\today}
\begin{document}
\begin{abstract}
In this article we define a class of fractional Orlicz-Sobolev spaces on Carnot groups and, in the spirit of the celebrated results of Bourgain-Brezis-Mironescu and of Maz'ya-Shaposhnikova, we study the asymptotic behavior of the Orlicz functionals when the fractional parameter goes to $1$ and $0$.
\end{abstract}
\maketitle
\section{Introduction}
In the seminal paper \cite{BBM}, Bourgain, Brezis and Mironescu proved that for any smooth bounded domain $\Omega\subset \mathbb{R}^n$, 
$u\in W^{1,p}(\Omega)$, $1\leq p<\infty$, the well-known fractional Gagliardo seminorm recovers its local counterpart  as $s$ goes to $1$, in the sense that
\begin{equation}\label{BBMsto1}
\lim_{s\uparrow 1}(1-s)\iint_{\Omega\times\Omega} \frac{|u(x)-u(y)|^p}{|x-y|^{n+sp}}\,dx\,dy = K(n,p)\int_\Omega |\nabla u|^p\,dx,
\end{equation}
\noindent where the constant $K(n,p)$ is defined as
\label{valoreK}
\begin{equation*}
K(n,p)=\frac{1}{p}\int_{{\mathbb S}^{n-1}}|{\boldsymbol \omega}\cdot h|^{p}\,d\mathcal{H}^{n-1}(h).
\end{equation*}
Here $\mathbb{S}^{n-1} \subset \mathbb{R}^n$ denotes the unit sphere, $\mathcal{H}^{n-1}$ is the
$(n-1)$-dimensional Hausdorff measure and ${\boldsymbol \omega}$ is an arbitrary unit vector of $\mathbb{R}^n$. 
See \cite{bre} for a survey on this and related results.\\
In the recent years, there has been an increasing interest towards this kind of results, aiming at extending it
in different directions. Despite the literature concerning the generalizations of such kind of results is now pretty vast,
let us try to give a brief account of it. 
The case of $BV$-functions, originally only partially answered in \cite{BBM}, has been considered independently in
\cite{davila} and \cite{Ponce}. In \cite{Spe,Spe2}, the authors covered the case of general open sets $\Omega \subset \mathbb{R}^{n}$, both for Sobolev and $BV$-functions.
More recently, starting with \cite{Ng06}, Nguyen obtained new characterizations of the classical Sobolev space by
means of more general nonlocal functionals, providing also Poincar\'{e}-type inequalities and several other results, see \cite{Ng08,Ng11,Ng11bis}. The nature of the nonlocal functionals, there considered, differs from the classical Gagliardo seminorm
and it has recently found applications in the field of image processing, see \cite{bre-linc,BHN,BHN2,BHN3}.\\
The above mentioned results have also been proved to hold in the case of {\it magnetic Sobolev spaces}. Roughly speaking, these spaces are the natural functional setting in electromagnetism when dealing with particles interacting with a magnetic field. 
We refer to \cite{SqVo} for the analogous of \eqref{BBMsto1} when $p=2$, and to \cite{PSV2} for the general case and for magnetic $BV$ functions.
We finally refer to \cite{NPSV,NPSV2} 
for similar results for more general nonlocal functionals akin to those considered in \cite{Ng06}.\\
A complementary and natural question is what happens when considering the limit as $s$ goes to $0$. 
The answer is contained in \cite{MS}, where it is proved that for any $n\geq 1$ and $p\in [1,\infty)$, 
\begin{equation}\label{MSclassic}
    \lim_{s\downarrow 0} s\iint_{\mathbb{R}^n\times\mathbb{R}^n}\frac{|u(x)-u(y)|^p}{|x-y|^{n+sp}}\,dx\,dy= 
    2\dfrac{|S|^{n-1}}{p}\int_{\mathbb{R}^n}|u|^p\,dx,
\end{equation}
whenever $u\in\bigcup_{s\in(0,1)}W^{s,p}_{0}(\mathbb{R}^n)$. We refer to \cite{PSV} for the magnetic version of \eqref{MSclassic}.
\medskip

The asymptotic theory described so far concerns always the classical Euclidean setting. Nevertheless,
recent contributions started to attack Bourgain-Brezis-Mironescu-type results even in the case where non--Euclidean geometries
appear. We refer, for instance, to \cite{KrMo} for the case of compact Riemannian manifolds.\\
As for many other problems, one of the non--Euclidean setting where to look for extensions is provided by Carnot groups, that are connected, simply connected and nilpotent Lie groups whose associated Lie algebra is stratified (see Section 2 for more details). These spaces are usually the easiest examples of
sub--Riemannian manifolds.
Fractional Sobolev spaces are now a well established notion even in Carnot groups, see \cite{FF,FS} for more details.
In this setting, in \cite{B,MaPi}, it was studied the validity of a Bourgain-Brezis-Mironescu-type formula to treat the asymptotic behavior of the corresponding fractional Sobolev seminorm as $s$ goes to $1$. 
We also refer to \cite{Cui} for the case of more general nonlocal functionals in the spirit of Nguyen.

\medskip 

Following this line of research, one of the most recent contributions, due to Fern\'andez Bonder and Salort, deals with Bourgain-Brezis-Mironescu-type results
for the class of Orlicz-Sobolev spaces, see also \cite{BS2} for similar results in the magnetic setting. 
The paper \cite{BS} is actually one of our motivations in this work.\\
Indeed, this manuscript aims at extending 
the validity of Bourgain-Brezis-Mironescu-type formulas in Carnot groups, when behaviors more general than powers are taken into account. In this context, Young functions play a preponderant role and Orlicz-Sobolev spaces become the natural framework to deal with. To fix ideas, when speaking of a Young function, we will refer to a continuous, non-negative, strictly increasing and convex function on $[0,\infty)$ vanishing at the origin, which, in order to give a well-posed space of definition, will be asked to fulfill the following structural growth condition
\begin{equation} \label{cond.intro} \tag{L}
1<p^-\leq \frac{t\,\varphi'(t)}{\varphi(t)} \leq p^+<\infty \quad \forall t>0.
\end{equation}
We refer the interested reader to the books \cite{DHHR,KR,PKJF} for a comprehensive introduction to Young functions and Orlicz spaces. Having these definitions in mind, following \cite{FS}, we define the natural generalization of the fractional Sobolev spaces for Carnot groups in the Orlicz setting. More precisely, given a Carnot group $\mathbb{G}$ of homogeneous dimension $Q$, a Young function $\varphi$ and a fractional parameter $0<s<1$, we consider the space
$$
W^{s,\varphi}(\mathbb{G}):=\{u:\mathbb{G}\to\mathbb{R}\ \text{measurable such that}\ \Phi_{\varphi}(u) + \Phi_{s,\varphi}(u)<\infty\}
$$
where
$$
\Phi_{\varphi}(u):=\int_{\mathbb{G}}\varphi(|u(x)|)\,dx,
\qquad \Phi_{s,\varphi}(u):=\iint_{\mathbb{G}\times\mathbb{G}}\varphi\left(\frac{|u(x)-u(y)|}{\|y^{-1}\cdot x\|_\mathbb{G}^s}\right)\,\frac{dx\,dy}{\|y^{-1}\cdot x\|_\mathbb{G}^Q}.
$$
It turns out that $W^{s,\varphi}(\mathbb{G})$ is a reflexive Banach space endowed with the correspondent Luxembourg norm (see Theorem \ref{2.10} for more details).

\medskip

We are now ready to state our first main result, that is a Bourgain-Brezis-Mironescu-type formula 
in the case of Orlicz-Sobolev spaces in Carnot groups.
We refer to Section \ref{section2} for a detailed account of 
all the definitions needed in the following
\begin{theorem}\label{BBM}
Let $\varphi$ be an Orlicz function such that the following limit exists
\begin{equation*}
    \tilde{\varphi}(t) := \lim_{s \uparrow 1}(1-s)\int_0^1\left(\int_{S} \varphi(t \|z'\|_{\mathbb{R}^m} r^{1-s})\,d\sigma(z)\right) \frac{dr}{r}.
\end{equation*}
Then, for any $u\in L^{\varphi}(\B{G})$ and $0<s<1$ it holds that
\[
\lim_{s\uparrow 1}(1-s)\Phi_{s,\varphi}(u) = \Phi_{\tilde{\varphi}}(\|\nabla_\mathbb{G} u\|_{\mathbb{R}^m}).
\]
\end{theorem}
Let us spend a few words about the proof of Theorem \ref{BBM}. Roughly speaking, the main technical point concerns a regularization argument in terms of truncations and convolutions (see Lemma \ref{2.13} and Lemma \ref{2.14}), combined with a compactness argument derived from a Rellich-Kondrachov-type result (see Theorem \ref{comp}).

Let us remark that, even in the prototype case $\varphi(t)=t^p$, we can prove Thorem \ref{BBM} only in the case of homogeneous norms that are invariant under horizontal rotations. See Section \ref{section2} for more details.
\medskip

In the second part of the paper we study the extension of \eqref{MSclassic} in the case of fractional Orlicz-Sobolev spaces on Carnot groups. Our main result in this context is provided by the following
\begin{theorem}\label{MS1}
If $u\in\bigcup_{s\in (0,1)}W^{s,\varphi}(\mathbb{G})$, then
\begin{equation*}
    \dfrac{4}{\C}\dfrac{QC_b}{p^+}\Phi_\varphi(u)\leq\liminf_{s\downarrow 0}s\Phi_{s,\varphi}(u)\leq\limsup_{s\downarrow 0}s\Phi_{s,\varphi}(u)\leq\C\dfrac{QC_b}{p^-}\Phi_\varphi(u).
\end{equation*}
In particular, if $\lim_{s\downarrow 0}s\Phi_{s,\varphi}(u)$ exists, then
\begin{align*}
    \dfrac{4}{\C}\dfrac{QC_b}{p^+}\Phi_\varphi(u)\leq\lim_{s\downarrow 0}s\Phi_{s,\varphi}(u)\leq\C\dfrac{QC_b}{p^-}\Phi_\varphi(u).
\end{align*}
\end{theorem}
Here $p^-,p^+$ are given in \eqref{L}, $C_b$ is the Lebesgue measure of the unit ball and $2<\C\leq 2^{p^+}$ (see Definition \ref{delta2} for details). We remind that, even if $p^-=p^+$, which actually corresponds to the classical fractional Sobolev space, 
Theorem \ref{MS1} would not provide a complete generalization of \eqref{MSclassic}, see Remark \ref{rmkmin}.
Due to this reason, we are interested in providing a better estimate on the limit, paying the price of
a more restrictive assumption on the Orlicz function $\varphi$.
\begin{theorem}\label{MS2}
Let us assume that the Orlicz function $\varphi$ satisfies the following Minkowski-type inequality 
\begin{equation}\label{Minkowski}\tag{M}
    \varphi^{-1}\left(\int_\mathbb{G}\varphi(|u(t)|+|v(t)|)\, dt\right)\leq\varphi^{-1}\left(\int_\mathbb{G}\varphi(|u(t)|)\, dt\right)+\varphi^{-1}\left(\int_\mathbb{G}\varphi(|v(t)|)\, dt\right)
\end{equation}
for any $u,v\in L^{\varphi}(\mathbb{G})$, where $\varphi^{-1}$ is the inverse map of $\varphi$. 
Then, for any $u\in W^{s,\varphi}(\mathbb{G})$
\begin{equation*}
    2\dfrac{QC_b}{p^+}\Phi_\varphi(u)\leq\liminf_{s\downarrow 0}s\Phi_{s,\varphi}(u)\leq\limsup_{s\downarrow 0}s\Phi_{s,\varphi}(u)\leq2\dfrac{QC_b}{p^-}\Phi_\varphi(u).
\end{equation*}
In particular, if $\lim_{s\downarrow 0}s\Phi_{s,\varphi}(u)$ exists, then
\begin{align*}
    2\dfrac{QC_b}{p^+}\Phi_\varphi(u)\leq\lim_{s\downarrow 0}s\Phi_{s,\varphi}(u)\leq 2\dfrac{QC_b}{p^-}\Phi_\varphi(u).
\end{align*}
\end{theorem}
We stress that the prototype $\varphi(t)=t^p$ trivially satisfies \eqref{Minkowski}.\\
As an immediate consequence, we obtain the following generalization of \eqref{MSclassic} in the framework of Carnot groups
\begin{cor}
If $u\in\bigcup_{s\in (0,1)}W^{s,\varphi}(\mathbb{G})$ with $\varphi(t)=t^p$, then $\lim_{s\downarrow 0}s\Phi_{s,\varphi}(u)$ exists and it holds that
\begin{equation}\label{Mazya}
    \lim_{s\downarrow 0}s\iint_{\mathbb{G}\times\mathbb
    {G}}\frac{|u(x)-u(y)|^p}{\|y^{-1}\cdot x\|_\mathbb{G}^{Q+sp}}\,dx\,dy=2\dfrac{QC_b}{p}\int_\mathbb{G}|u|^p\,dx.
\end{equation}
\end{cor}
We want to stress another interesting aspect: we are able to prove Theorem \ref{BBM}, Theorems \ref{MS1} and Theorem \ref{MS2} if we work with an homogeneous norm which satisfies the classical triangular inequality
\begin{equation*}
        \big|\|y\|_\mathbb{G}-\|x\|_\mathbb{G}\big|\leq\|y^{-1}\cdot x\|_\mathbb{G}\leq\|x\|_\mathbb{G}+\|y\|_\mathbb{G},
\end{equation*}
which is not in general true by a homogeneous norm. See Section \ref{section2} for more details and examples.
\medskip

Finally, we want to recall that the asymptotic behavior of the perimeter functional has also been addressed, 
either in Euclidean contexts as well as in Carnot settings, see e.g. \cite{AmbDepMart, CV, DiFiPaVa, FMPPS, Ludwig1, Ludwig2}.

\medskip

\indent The structure of the paper is the following: in Section 2, we provide the basic necessary notions of Carnot groups, Young functions and Orlicz spaces. In Section 3, the proof of Theorem \ref{BBM} is given and, finally, in Section 4, we conclude by proving Theorem \ref{MS1} and Theorem \ref{MS2}.
\section{Preliminaries}\label{section2}
\subsection{Carnot groups}
We start this section recalling the basic notions of Carnot groups.

A Carnot group $\mathbb{G}= (\mathbb{R}^n,\cdot)$ is a connected, simply connected 
and nilpotent Lie group, whose Lie algebra $\mathfrak{g}$ admits a stratification. Namely, there exist linear subspaces, usually called {\it layers}, such that
$$\mathfrak{g}=V_1\oplus..\oplus V_k, \quad [V_1,V_i]=V_{i+1}, \quad V_{k}\neq \{0\}, \quad V_i=\{0\} \, \textrm{if $i>k$},$$
\noindent where $k$ is usually called the {\it step} of the group $(\mathbb{G},\cdot)$ and
$$[V_i,V_j]:=\textrm{span}\left\{[X,Y]: \, X\in V_i,Y\in V_j\right\}.$$

The explicit expression of the group law $\cdot$ can be deduced from the Hausdorff-Campbell formula, see e.g. \cite{BLU}.
The group law can be used to define a diffeomorphism, usually called {\it left--translation} 
$\gamma_y : \mathbb{G} \to \mathbb{G}$ for every $y\in \mathbb{G}$, defined as
$$\gamma_y (x) := y \cdot x \quad \textrm{for every } x \in \mathbb{G}.$$
A Carnot group $\mathbb{G}$ is also endowed with a family of automorphisms of the group $\delta_\lambda:\mathbb{G}\to\mathbb{G}$, $\lambda\in\mathbb{R}^+$, called {\it dilations},
given by
\begin{equation*}
   \delta_\lambda(x_1,\ldots,x_n):=(\lambda^{d_1}x_1,..,\lambda^{d_n}x_n),
\end{equation*}
\noindent where $(x_1,\ldots, x_n)$ are the {\it exponential coordinates} of $x \in \mathbb{G}$,
$d_{j}\in \mathbb{N}$ for every $j=1,\ldots,n$ and $1=d_1=\ldots=d_m < d_{m+1} \leq \ldots \leq d_{n}$
for $m := {\rm dim}(V_1)$. Here the group $\mathbb{G}$ and the algebra $\mathfrak{g}$ are identified through the exponential mapping.
The $n$-dimensional Lebesgue measure $\mathcal{L}^{n}$ of $\mathbb{R}^n$ 
provides the Haar measure on $\mathbb{G}$, see e.g. \cite[Proposition 1.3.21]{BLU}.

It is customary to denote with $Q:=\sum_{i=1}^k i\ {\rm dim}(V_i)$ the homogeneous dimension of $\mathbb{G}$ 
which corresponds to the Hausdorff dimension of $\mathbb{G}$ (w.r.t. an appropriate sub--Riemannian distance, see below). 
This is generally greater than (or equal to) the topological dimension of $\mathbb{G}$ and it coincides with it
only when $\mathbb{G}$ is the Euclidean group $(\mathbb{R}^n,+)$, which is the only Abelian Carnot group.\\

Carnot groups are also naturally endowed with sub-Riemannian distances which make them interesting examples of metric spaces. 
A first well--known example of such metrics is provided by the Carnot-Carath\'eodory distance $d_{cc}$, see e.g. \cite[Definition 5.2.2]{BLU}, which
is a path--metric resembling the classical Riemannian distance.
In our case, we will work with metrics induced by homogeneous norms.
\begin{defn}
A homogeneous norm $\|\cdot\|_\mathbb{G}:\mathbb{G}\to\mathbb{R}^+_0$ is a continuous function with the following properties:
\begin{itemize}
    \item [$(i)$] $\|x\|_\mathbb{G}=0\ \textrm{if and only if } x=0$ for every $x\in\mathbb{G}$;
    \item [$(ii)$] $\|x^{-1}\|_\mathbb{G}=\|x\|_\mathbb{G}$ for every $x\in\mathbb{G}$;
    \item [$(iii)$] $\|\delta_\lambda x\|_\mathbb{G}=\lambda \|x\|_\mathbb{G}$ for every $\lambda\in\mathbb{R}^+$ and for every $x\in\mathbb{G}$.
\end{itemize}
\end{defn}
We remind that any homogeneous norm induces a left--invariant homogeneous distance by
\begin{equation*}
    d(x,y):=\|y^{-1}\cdot x\|_\mathbb{G}\quad\textrm{for every }x,y\in\mathbb{G}.
\end{equation*}
A concrete example of such kind of homogeneous distance is given by the Kor\'anyi distance, see e.g. \cite{Cygan}. From now on, we will write $B(x,\eps)$ to denote the ball of center $x\in\mathbb{G}$ and radius $\eps>0$ 
w.r.t the distances $d$.
\medskip

In the proceeding of the paper, however, we will ask for the following stronger hypothesis on the norm $\|\cdot\|_\mathbb{G}$:
\begin{itemize}
    \item[$(iv)$] invariance under horizontal rotations;
    \item[$(v)$] the validity of the classical triangular inequality
    \begin{equation*}
        \big|\|y\|_\mathbb{G}-\|x\|_\mathbb{G}\big|\leq\|y^{-1}\cdot x\|_\mathbb{G}\leq\|x\|_\mathbb{G}+\|y\|_\mathbb{G}.
    \end{equation*}
\end{itemize}
An example of such kind of norm, whose induced distance is equivalent to the well-known Carnot-Carath\'eodory distance, is given in \cite{FSSC,FSSC2}.

\begin{prop}
Let $\B{G}$ be a Carnot group and $f\in L^1(\B{G})$. Then the Haar measure on $\B{G}$
\begin{itemize}
\item[(i)] is invariant under left and right translations:
$$
\int_\mathbb{G} f(x)\,dx = \int_\mathbb{G} f(x\cdot y)\,dx =\int_\mathbb{G} f(y\cdot x)\,dx  \quad \forall y\in\B{G}
$$
\item[(ii)] scales under group dilations by the homogeneous dimension of $\B{G}$:
$$
\int_\mathbb{G} f(\delta_\lambda x)\,dx = \lambda^Q \int_\mathbb{G} f(x)\,dx \quad \forall \lambda>0.
$$ 
\end{itemize}
\end{prop}
It trivially follows that $|B(x,r)|=r^Q|B|=r^Q C_b$ for all $x\in \B{G}$ and $r>0$, where $B=B(0,1)$ and $C_b$ denotes its Lebesgue measure.
\medskip

The following three Propositions will be very useful in the sequel.
\begin{prop}\cite[Proposition 1.13]{FS}\label{fs_1.13}
Let $f\in L^1_{\rm{loc}}(\mathbb{G}\setminus\{0\})$ be an homogeneous function of degree $-Q$, i.e., $f(\delta_\lambda x)=\lambda^{-Q}f(x)$. Then, there exists a constant $M_f$, mean value of $f$, such that
\begin{equation*}
    \int_\mathbb{G}f(x)g(\|x\|_\mathbb{G})\,dx=M_f\int_0^{+\infty}g(r)\frac{dr}{r}
\end{equation*}
for any $g\in L^1(\mathbb{R}^+,\frac{dr}{r})$.
\end{prop}
As a consequence of the previous result, we are able to compute explicitly integrals on balls, of functions depending only on the distance from the center of the ball, in terms of integrals on the real line.
\begin{prop}\label{prop.radialfunct}
Let $f\in L^1(\mathbb{R}^+)$ and $R>0$. Then
\[
\int_{B(y,R)}f(\|y^{-1}\cdot x\|_\mathbb{G})\,dx=\int_{B(0,R)} f(\|x\|_\mathbb{G})\,dx=QC_b\int_0^R r^{Q-1}f(r)\,dr
\]
and
\[
\int_{\mathbb{G}\setminus B(y,R)}f(\|y^{-1}\cdot x\|_\mathbb{G})\,dx=\int_{\mathbb{G}\setminus B(0,R)} f(\|x\|_\mathbb{G})\,dx=QC_b\int_R^{+\infty} r^{Q-1}f(r)\,dr.
\]
\end{prop}
\begin{proof}
At first, let us compute the constant $M_f$ for the function $f(x)=\|x\|_\mathbb{G}^{-Q}$.\\
By Proposition \ref{fs_1.13}, taking $g(\|x\|_\mathbb{G})=\|x\|_\mathbb{G}^Q\chi_{[0,1]}(\|x\|_\mathbb{G})$, we get
\begin{equation*}
    C_b=\int_B\,dx=\int_\mathbb{G}\|x\|_\mathbb{G}^{-Q}\|x\|_\mathbb{G}^Q\chi_{[0,1]}(\|x\|_\mathbb{G})\,dx=M_{\|x\|_\mathbb{G}^{-Q}}\int_0^1 r^{Q-1}\,dr=\frac{M_{\|x\|_\mathbb{G}^{-Q}}}{Q},
\end{equation*}
i.e., $M_{\|x\|_\mathbb{G}^{-Q}}=QC_b$.

Therefore, still by Proposition \ref{fs_1.13}, we have
\begin{align*}
    \int_{B(0,R)} f(\|x\|_\mathbb{G})\,dx&=\int_\mathbb{G}\|x\|_\mathbb{G}^{-Q}\|x\|_\mathbb{G}^Q\chi_{[0,R]}(\|x\|_\mathbb{G})f(\|x\|_\mathbb{G})\,dx\\
    &=M_{\|x\|_\mathbb{G}^{-Q}}\int_0^R r^{Q-1}f(r)\,dr=QC_b\int_0^R r^{Q-1}f(r)\,dr
\end{align*}
and
\begin{align*}
    \int_{\mathbb{G}\setminus B(0,R)} f(\|x\|_\mathbb{G})\,dx&=\int_\mathbb{G}\|x\|_\mathbb{G}^{-Q}\|x\|_\mathbb{G}^Q\chi_{[R,+\infty[}(\|x\|_\mathbb{G})f(\|x\|_\mathbb{G})\,dx\\
    &=QC_b\int_R^{+\infty} r^{Q-1}f(r)\,dr.
\end{align*}
\end{proof}
\begin{prop}\cite[Proposition 1.15]{FS}\label{prop.fs}
There exists a unique Radon measure $\sigma$ on $S$ such that for all $u\in L^1(\B{G})$
\[
\int_{\B{G}}u(x)dx = \int_0^{+\infty}\left(\int_S u(\delta_r z)r^{Q-1}\,d\sigma(z)\right)\,dr
\]
where $S$ is the unit sphere in $\mathbb{G}$.
\end{prop}
We conclude this part recalling the notion of Pansu differentiability, given by Pansu in \cite{P}. We remind that, for any $u:\mathbb{G} \to \mathbb{R}$ sufficiently smooth, given $(X_1,..,X_m)$ a basis of the horizontal layer $V_1$, of left-invariant vector fields, then the {\it horizontal gradient} of $u:\mathbb{G} \to \mathbb{R}$ is defined as
$$\nabla_\mathbb{G}u:=\sum_{j=1}^m(X_j u)X_j=(X_1u,..,X_m u).$$
\begin{defn}
A function $f:\B{G}\to\B{R}$ is said to be \emph{Pansu differentiable} at $x\in\B{G}$ if there exists a $\B{G}-$linear map $L_x^f:\B{G}\to\B{R}$, named Pansu differential, such that
\begin{equation*}
\lim_{\|h\|_\mathbb{G}\to0}\frac{f(x\cdot h)-f(x)-L_x^f(h')}{\|h\|_\mathbb{G}}=0.
\end{equation*}
We will say that $f$ is Pansu differentiable in $\B{G}$ if it is Pansu differentiable at any $x\in\B{G}$.
\end{defn}
\begin{rmk}
Let us notice that the Pansu differential $L_x^f$ does not depend on the basis of the Lie algebra $\mathfrak{g}$. In the sequel, we will use the notation
\[
L_x^f(h') =  \nabla_\mathbb{G} f \cdot h'.
\]
As noted in \cite[Section 5]{FSSC} every function in $C^1(\B{G})$ is also Pansu differentiable and the inclusion is actually strict.
\end{rmk}
\subsection{Orlicz functions}
\begin{defn}Let $\phi:\mathbb{R}^+_0\to\mathbb{R}^+_0$ be a real valued function such that:
\begin{itemize} 
    \item[$(i)$] $\phi(0)=0$, and $\phi(t)>0$ for any $t>0$;
    \item[$(ii)$] $\phi$ is nondecreasing on $\mathbb{R}^+_0$;
    \item[$(iii)$] $\phi$ is right-continuous in $\mathbb{R}^+_0$ and $\lim_{t\to+\infty}\phi(t)=+\infty$.
\end{itemize}
Then, the real valued function defined on $\mathbb{R}^+_0$ by 
$$\varphi(t)=\int_0^t \phi(s)\,ds$$
is called Orlicz function (or Nice Young function).
\end{defn}
It is easy to show that hypothesis $(i)-(iii)$ imply that $\varphi$ is continuous, Locally Lipschitz continuous, strictly increasing and convex on $\mathbb{R}^+_0$. Moreover, $\varphi(0)=0$ and $\varphi$ is superlinear at zero and at infinity, i.e., 
$$\lim_{t\to 0^+}\frac{\varphi(t)}{t}=0 \qquad \lim_{t\to+\infty}\frac{\varphi(t)}{t}=+\infty.$$
Up to normalization, we can assume $\varphi(1)=1$. Hypothesis $(i)-(iii)$ also guarantee the existence of $\varphi^{-1}:\mathbb{R}^+_0\to\mathbb{R}^+_0$, which is continuous, concave and strictly increasing, with $\varphi^{-1}(0)=0$ and $\varphi^{-1}(1)=1$.

From now on, the following growth condition will be required on $\varphi$:
\begin{equation} \label{L} \tag{L}
p^-\leq \frac{t\,\phi(t)}{\varphi(t)} \leq p^+ \quad \forall t>0,
\end{equation}
where $p^-\leq p^+$ are positive constants grater than $1$. It holds that
\begin{align*}
  &s^{\overline{p}}\varphi(t)\leq\varphi(st)\leq s^{\tilde{p}}\varphi(t),\tag{$\varphi_1$}\label{G1}\\
  &\varphi(s+t)\leq\frac{2^{p^+}}{2}(\varphi(s)+\varphi(t)).\tag{$\varphi_2$}\label{G2}
 \end{align*}
for any $s,t\in\mathbb{R}^+_0$, where $s^{\tilde{p}}:=\max\{s^{p^-},s^{p^+}\}$ and $s^{\overline{p}}:=\min\{s^{p^-},s^{p^+}\}$.

Let us notice that $p^-=p^+$ if and only if $\phi(t)=t^p$, being $\varphi(1)=1$.
\\
We remind that the conjugate function of $\varphi$, defined as its Legendre's transform, is
\begin{align*}
    \varphi^*(s):=\sup_{t>0}\{ st-\varphi(t)\}.
\end{align*}
\begin{defn}\label{delta2}
The smallest $C\in\mathbb{R}^+$ such that the following $\Delta_2$-condition holds
\begin{align*}
\varphi(2t)\leq C\varphi(t)\quad \forall t\in\mathbb{R}^+_0,   
\end{align*}
is called the $\Delta_2$-constant and it is denoted by $\C$. By $\eqref{G2}$, we have that
\begin{equation}\label{doubling}
    2<\C\leq 2^{p^+}.
\end{equation}
It is not difficult to show that \eqref{L} is equivalent to require the $\Delta_2$-condition both on $\varphi$ and $\varphi^*$ (see for instance \cite[Chapter 4]{PKJF}).
\end{defn}
The following Lemma can be seen as an improvement of \eqref{G2}.
\begin{lem}\cite[Lemma 2.6]{BS}\label{lem.2.6}
Let $\varphi$ be an Orlicz function and let $s,t\in\mathbb{R}^+_0$. Then, for any $\delta>0$, there exists a positive constant $C_\delta$ such that
$$
\varphi(s+t)\leq C_\delta\varphi(s)+(1+\delta)^{p^+} \varphi(t).
$$
\end{lem} 
We conclude this section recalling a fundamental definition which is the natural counterpart of \cite[Remark 2.15]{BS} in the context of Carnot groups. From now on, when necessary, a generic $z\in\B{G}$ will be denoted as $z=(z',z'')$ where $z'=(z_1,..,z_m)$ is the horizontal part and $z''=(z_{m+1},..,z_n)$ is the vertical one.
\begin{defn}
For an Orlicz function $\varphi$ and $t\in\mathbb{R}^+$, we define the bounded function
\begin{equation*}
    \tilde{\varphi}^+(t) \coloneqq \limsup_{s \uparrow 1}(1-s)\int_0^1\left(\int_{S} \varphi(t \|z'\|_{\mathbb{R}^m} r^{1-s})\,d\sigma(z)\right) \frac{dr}{r}.
\end{equation*}
A similar definition with $\liminf$ instead of $\limsup$ is used to define $\tilde{\varphi}^-$. When they coincide, we will define
\begin{equation}\label{eq.phitilde}
    \tilde{\varphi}(t) \coloneqq \lim_{s \uparrow 1}(1-s)\int_0^1\left(\int_{S} \varphi(t \|z'\|_{\mathbb{R}^m} r^{1-s})\,d\sigma(z)\right) \frac{dr}{r}.
\end{equation}
\end{defn}
\begin{prop}\label{2.16}
The functions $\tilde{\varphi}^\pm$ are still Orlicz functions, both of them equivalent to $\varphi$, i.e., there exist $c_1,c_2>0$ such that
$$c_1\varphi(t)\leq\tilde{\varphi}^\pm(t)\leq c_2\varphi(t)$$
for any $t\in\mathbb{R}^+$.
\end{prop}
\begin{proof}
$\tilde{\varphi}^\pm$ are Orlicz functions by similar arguments of \cite[Proposition 2.16]{BS}. Moreover, by \eqref{G1}, we can notice that
\begin{equation*}
    \begin{split}
        \int_0^1\left(\int_{S} \varphi(t\|z'\|_{\mathbb{R}^m}r^{1-s})\,d\sigma(z)\right)\frac{dr}{r} &\leq \int_{S} \|z'\|_{\mathbb{R}^m}^{p^-} d\sigma(z) \int_0^1 \varphi(tr^{1-s})\frac{dr}{r}\\
        &\leq QC_b\varphi(t)\int_0^1 r^{(1-s)p^- -1}dr = \frac{QC_b}{(1-s)p^-}\varphi(t)
    \end{split}
\end{equation*}
and
\begin{equation*}
    \begin{split}
        \int_0^1\left(\int_{S} \varphi(t\|z'\|_{\mathbb{R}^m}r^{1-s})\,d\sigma(z)\right)\frac{dr}{r} &\geq \int_{S} \|z'\|_{\mathbb{R}^m}^{p^+} d\sigma(z) \int_0^1 \varphi(tr^{1-s})\frac{dr}{r}\\
        &\geq QC_b\varphi(t)\int_0^1 r^{(1-s)p^+ -1}dr = \frac{QC_b}{(1-s)p^+}\varphi(t).
    \end{split}
\end{equation*}
Thus, taking $c_1:=\frac{QC_b}{p^+}$ and $c_2:=\frac{QC_b}{p^-}$, we get the thesis.
\end{proof}
\begin{rmk}
Let us notice that $c_1=c_2=\frac{QC_b}{p}$ if and only if $\varphi(t)=t^p$. We also remind that explicit examples of $\tilde{\varphi}$, in the Euclidean case, are given in \cite[Example 2.17]{BS}.
\end{rmk}
\subsection{The functional setting}
\begin{defn}
Let $\mathbb{G}$ be a Carnot group, let $\varphi$ be an Orlicz function and let $0<s\leq 1$. We define, with a little abuse of notation, the Orlicz-Lebesgue space and the Fractional Orlicz-Sobolev spaces, respectively, as
\begin{align*}
L^{\varphi}(\mathbb{G})&\coloneqq\{u:\mathbb{G}\to\mathbb{R}\ \text{measurables such that}\ \Phi_{\varphi}(u)<\infty\}\\
W^{s,\varphi}(\mathbb{G})&\coloneqq\{u\in L^{\varphi}(\mathbb{G})\ \text{such that}\ \Phi_{s,\varphi}(u)<\infty\},
\end{align*}
where
\begin{align*}
\Phi_{\varphi}(u)&\coloneqq\int_{\mathbb{G}}\varphi(|u(x)|)\,dx,\\
\Phi_{s,\varphi}(u)&\coloneqq
\begin{cases}
\displaystyle{\iint_{\mathbb{G}\times\mathbb{G}}\varphi\left(\frac{|u(x)-u(y)|}{\|y^{-1}\cdot x\|_\mathbb{G}^s}\right)\,\frac{dx\,dy}{\|y^{-1}\cdot x\|_\mathbb{G}^Q}}&\qquad\text{if}\ 0<s<1\\
\displaystyle \Phi_\varphi(\|\nabla_\mathbb{G} u\|_{\mathbb{R}^m})&\qquad\text{if}\ s=1
\end{cases}.
\end{align*}
\end{defn}
These spaces are usually endowed with the so-called Luxemburg norms, studied by Luxemburg in \cite{L}, and defined as
\begin{align*}
\|u\|_{\varphi}&\coloneqq\inf\{\lambda>0:\Phi_{\varphi}\left(\frac{u}{\lambda}\right)\leq 1\}\\
\|u\|_{s,\varphi}&\coloneqq\|u\|_{\varphi}+[u]_{s,\varphi}
\end{align*}
where
$$[u]_{s,\varphi}\coloneqq\inf\{\lambda>0:\Phi_{s,\varphi}\left(\frac{u}{\lambda}\right)\leq 1\}$$
is the $(s,\varphi)-$Gagliardo seminorm.

By well-known results given in \cite{DHHR,HHK} for the Euclidean case, it is easy to characterize these spaces as follows
\begin{theorem}\label{2.10}
Let $\varphi$ be an Orlicz function, then $L^{\varphi}(\B{G})$ and $W^{1,\varphi}(\B{G})$ are separable Banach spaces. Moreover, if both $\varphi$ and $\varphi^*$ satisfy the $\Delta_2$-condition, then the spaces $L^{\varphi}(\B{G})$ and $W^{1,\varphi}(\B{G})$ are also reflexive and the dual space of $L^{\varphi}(\B{G})$ can be identified with $L^{\varphi^*}(\B{G})$. Finally, $C_c^{\infty}(\B{G})$ is dense in both $L^{\varphi}(\B{G})$ and $W^{1,\varphi}(\B{G})$.
\end{theorem}
The proof of Theorem \ref{2.10} trivially follows from the Euclidean case. The reader can see from instance \cite[Theorem 2.3.13, Theorem 2.5.10]{DHHR} and \cite[Theorem 5.3, Theorem 5.5, Corollary 3.7, Corollary 3.9]{HHK}, where a more general theory is treated.

\medskip
Following the same technique of \cite[Proposition 2.11]{BS}, we can also state the following Theorem.
\begin{theorem}
Let us assume the same hypothesis of the previous Theorem. Then, for each $s\in (0,1)$, the space $W^{s,\varphi}(\B{G})$ is a reflexive and separable Banach space. Moreover, $C_c^{\infty}(\B{G})$ is dense in $W^{s,\varphi}(\B{G})$.
\end{theorem}
As in the Euclidean case, the immersion of the space $W^{s,\varphi}(\B{G})$ into $L^\varphi(\mathbb{G})$ is compact, as a consequence of the following
\begin{theorem}\cite[Theorem 11.4]{KR}\label{11.4}
    Any sequence of functions $\{v_k\}_k\subset L^{\varphi}(\mathbb{G})$ is compact if and only if the following two conditions are satisfied:
    \begin{itemize}
        \item[$(i)$] $\Phi_\varphi(v_k)$ is bounded;
        \item[$(ii)$] for any $\varepsilon>0$ there exists $\delta>0$ such that $\Phi_\varphi(\tau_h v_k-v_k)<\varepsilon$ for any $h\in\mathbb{G}$ such that $\|h\|_\mathbb{G}<\delta$, where $\tau_h u(x)\coloneqq u(x\cdot h)$ for any $x\in\mathbb{G}$.
    \end{itemize}
\end{theorem}
\begin{theorem}\label{comp}
Let $0<s<1$ and $\varphi$ be an Orlicz function. Then, from every bounded sequence $\{u_n\}_n\subset W^{s,\varphi}(\B{G})$, there exist $u\in W^{s,\varphi}(\B{G})$ and $\{u_{n_k}\}_k\subset\{u_n\}_n$ such that $u_{n_k}\to u$ in $L^{\varphi}(\B{G})$.
\end{theorem}
\begin{proof}
Let us fix $u\in W^{s,\varphi}(\B{G})$. In order to apply Theorem \ref{11.4}, we want to show the existence of a constant $M>0$ such that
\begin{equation}\label{fkt}
    \Phi_{\varphi}(\tau_hu - u) \leq M\|h\|_\mathbb{G}^{sp^-}\Phi_{s,\varphi}(u)
\end{equation}
for every $h\in\mathbb{G}$ such that $\|h\|_\mathbb{G}<\frac{1}{2}$.

For any $y\in B(x,\|h\|_\mathbb{G})$, by the monotonicity of $\varphi$, the $\Delta_2$-condition and being $|B(x,r)|=r^Q|B|=r^Q C_b$, we have
\begin{equation}\label{eq.lemma3.2}
    \begin{split}
        \Phi_{\varphi}(\tau_h u-u)&=\int_\mathbb{G}\varphi(|u(x\cdot h)-u(y)+u(y)-u(x)|)\,dx\\
        &\leq\frac{\C}{2}\left[\int_\mathbb{G}\varphi(|u(x\cdot h)-u(y)|)\,dx+\int_\mathbb{G}\varphi(|u(y)-u(x)|)\,dx\right]\\
        &=\frac{\C}{2C_b\|h\|_\mathbb{G}^Q}\int_{B(x,\|h\|_\mathbb{G})}\left(\int_\mathbb{G}\varphi(|u(x\cdot h)-u(y)|)\,dx\right)\,dy\\
        &\quad+\frac{\C}{2C_b\|h\|_\mathbb{G}^Q}\int_{B(x,\|h\|_\mathbb{G})}\left(\int_\mathbb{G}\varphi(|u(y)-u(x)|)\,dx\right)\,dy\\
        &=\frac{\C}{2C_b\|h\|_\mathbb{G}^Q}(I_1 + I_2).
    \end{split}
\end{equation}
Let us notice that, by the triangular inequality, $$\|y^{-1}\cdot x\cdot h\|_\mathbb{G}\leq\|y^{-1}\cdot x\|_\mathbb{G}+\|h\|_\mathbb{G}\leq 2\|h\|_\mathbb{G}.$$
Therefore, by \eqref{G1}, the monotonicity of $\varphi$ and a change of variables, we have
\begin{equation*}
    \begin{split}
        I_1&=\int_{B(x,\|h\|_\mathbb{G})}\left(\int_\mathbb{G}\varphi\left(\frac{|u(x\cdot h)-u(y)|}{\|y^{-1}\cdot x\cdot h\|_\mathbb{G}^s}\|y^{-1}\cdot x\cdot h\|_\mathbb{G}^s\right)\frac{\|y^{-1}\cdot x\cdot h\|_\mathbb{G}^Q}{\|y^{-1}\cdot x\cdot h\|_\mathbb{G}^Q}\,dx\right)\,dy\\
        &\leq2^Q\|h\|_\mathbb{G}^Q\int_{B(x,\|h\|_\mathbb{G})}\left(\int_\mathbb{G}\varphi\left(\frac{|u(x\cdot h)-u(y)|}{\|y^{-1}\cdot x\cdot h\|_\mathbb{G}^s}(2\|h\|_\mathbb{G})^s\right)\frac{dx}{\|y^{-1}\cdot x\cdot h\|_\mathbb{G}^Q}\right)\,dy\\
        &\leq2^{sp^-+Q}\|h\|_\mathbb{G}^{sp^-+Q}\int_{B(x,\|h\|_\mathbb{G})}\left(\int_\mathbb{G}\varphi\left(\frac{|u(z)-u(y)|}{\|y^{-1}\cdot z\|_\mathbb{G}^s}\right)\frac{dz}{\|y^{-1}\cdot z\|_\mathbb{G}^Q}\right)\,dy\\
        &\leq2^{sp^-+Q}\|h\|_\mathbb{G}^{sp^-+Q}\Phi_{s,\varphi}(u).
    \end{split}
\end{equation*}
Similarly
$$I_2\leq\|h\|_\mathbb{G}^{sp^-+Q}\Phi_{s,\varphi}(u).$$
Thus, by \eqref{eq.lemma3.2}, we finally have
$$\Phi_{\varphi}(\tau_h u-u)\leq\frac{\C}{2C_b}(2^{sp^-+Q}+1)\|h\|_\mathbb{G}^{sp^-}\Phi_{s,\varphi}(u):=M\|h\|_\mathbb{G}^{sp^-}\Phi_{s,\varphi}(u).$$
\medskip

Let now $\{u_n\}_n\subset W^{s,\varphi}(\B{G})$ be a bounded sequence in $W^{s,\varphi}(\B{G})$. In particular, $\{u_n\}_n$ is bounded in $L^{\varphi}(\B{G})$. Therefore, by \eqref{fkt}
\[
\sup_{n\in\B{N}}\Phi_\varphi(\tau_h u_n - u_n)\leq \sup_{n\in\B{N}}(\Phi_{s,\varphi}(u_n) + \Phi_{\varphi}(u_n))M\|h\|_\mathbb{G}^{sp^-}.
\]

Thus, by Theorem \ref{11.4}, there exist $u\in L^{\varphi}(\B{G})$ and $\{u_{n_k}\}_k\subset\{u_n\}_n$ such that $u_{n_k}\to u$ in $L^{\varphi}(\B{G})$.
In order to conclude the proof, we show that $u\in W^{s,\varphi}(\B{G})$.

By the Fatou's Lemma and the continuity of $\varphi$, we have
\begin{align*}
    \Phi_{s,\varphi}(u)&=\iint_{\mathbb{G}\times\mathbb{G}}\varphi\left(\frac{|u(x)-u(y)|}{\|y^{-1}\cdot x\|_\mathbb{G}^s}\right)\frac{dx\,dy}{\|y^{-1}\cdot x\|_\mathbb{G}^Q}\\
    &\leq\liminf_{k\to\infty} \iint_{\mathbb{G}\times\mathbb{G}}\varphi\left( \frac{|u_{n_k}(x)-u_{n_k}(y)|}{\|y^{-1}\cdot x\|_\mathbb{G}^s}\right)\frac{dx\,dy}{\|y^{-1}\cdot x\|_\mathbb{G}^Q}\\
    &\leq\sup_{n\in\B{N}}\Phi_{s,\varphi}(u_{n_k})<\infty.
\end{align*}
\end{proof}
Taking into account \cite[Definition 5.3.6]{BLU}, we state below the notions of convolution and truncation on Carnot groups.
\begin{defn}\label{def_moll}
Let $\rho\in C^{\infty}_c(\B{G};\mathbb{R}_0^+)$ be the standard mollifier, that is, $\supp(\rho)\subset B(0,1)$ and $\int_\mathbb{G}\rho(x)\,dx = 1$. Then, for any $u\in L^{\varphi}(\B{G})$ and $x\in\B{G}$, considering the sequence of mollifiers $$\rho_{\eps}(x)\coloneqq\eps^{-Q}\rho(\delta_{\eps^{-1}}x)\qquad \varepsilon>0,$$ 
we define the regularized functions of $u$, $\{u_\varepsilon\}_\varepsilon\subset L^\varphi(\mathbb{G})\cap C^\infty(\mathbb{G})$, as
\begin{equation*}
    \begin{split}
        u_{\eps}(x)\coloneqq(u*\rho_{\eps})(x)&\coloneqq\int_\mathbb{G}u(y)\rho_\eps(x\cdot y^{-1})\,dy=\int_{B(0,\varepsilon)}u(y^{-1}\cdot x)\rho_{\eps}(y)\,dy\\
        &=\int_{B(0,1)}u((\delta_\eps z)^{-1}\cdot x)\rho(z)\,dz
    \end{split}
\end{equation*}
for any $\varepsilon>0$.
\end{defn}
\begin{defn}\label{def_trun}
Given $\eta \in C^{\infty}_c(\B{G})$ such that $0\leq \eta\leq 1$, $\eta = 1$ in $B(0,1)$, $\supp(\eta)\subset B(0,2)$ and $\|\nabla_\mathbb{G}\eta\|_{\mathbb{R}^m}\leq 2$, we define the cut-off functions
$$\eta_k(x)\coloneqq \eta(\delta_{k^{-1}}x)$$
for any $k\in\B{N}$. Let us notice that $0\leq \eta_k\leq 1$, $\eta_k = 1$ in $B(0,k)$, $\supp(\eta_k)\subset B(0,2k)$ and $\|\nabla_\mathbb{G}\eta_k\|_{\mathbb{R}^m}\leq \frac{2}{k}$.

For any $u\in L^{\varphi}(\B{G})$ we define the truncated functions of $u$, $\{u_k\}_k$, as
\[
u_k \coloneqq \eta_k u
\]
for any $k\in\mathbb{N}$. We remind that $supp(u_k)\subset B(0,2k)$.
\end{defn}
The two following Lemmas will be useful in the next section.
\begin{lem}\label{2.13}
Let $u\in L^\varphi(\mathbb{G})$ and let $\{u_\varepsilon\}_\varepsilon$ be a sequence of regularized functions of $u$, in the sense of Definition \ref{def_moll}. Then
\begin{equation*}
\Phi_{s,\varphi}(u_\varepsilon)\leq\Phi_{s,\varphi}(u)
\end{equation*}
for any $\varepsilon>0$ and $0<s<1$.
\end{lem}
\begin{proof}
Let $x,y\in\mathbb{G}$ and let $h=y^{-1}\cdot x$. Then, by the Jensen's inequality and the monotonicity of $\varphi$, we have
\begin{equation*}
    \begin{split}
        \varphi\left(\frac{|u_{\eps}(x\cdot h) - u_{\eps}(x)|}{\|h\|_\mathbb{G}^s}\right)&\leq \,\varphi\left( \int_{B(0,\varepsilon)}\frac{|u(y^{-1}\cdot x\cdot h)-u(y^{-1}\cdot x)|}{\|h\|_\mathbb{G}^s}\rho_{\eps}(y)\,dy \right)\\
        &\leq\int_{B(0,\varepsilon)}\varphi\left(\frac{|u(y^{-1}\cdot x\cdot h)-u(y^{-1}\cdot x)|}{\|h\|_\mathbb{G}^s}\right)\rho_{\eps}(y)\,dy. 
    \end{split}
\end{equation*}
Therefore, by the invariance of the norm under translations, we have
\begin{equation*}\label{eq.mollifiers}
    \begin{split}
        &\int_\mathbb{G}\varphi\left( \frac{|u_{\eps}(x\cdot h) - u_{\eps}(x)|}{\|h\|_\mathbb{G}^s} \right)\frac{dx}{\|h\|_\mathbb{G}^Q}\\
        &\leq\int_\mathbb{G}\left(\int_{B(0,\varepsilon)} \varphi\left(\frac{|u(y^{-1}\cdot x\cdot h)-u(y^{-1}\cdot x)|}{\|h\|_\mathbb{G}^s} \right)\rho_{\eps}(y)\,dy  \right)\frac{dx}{\|h\|_\mathbb{G}^Q}\\
        &=\int_\mathbb{G}\left(\int_\mathbb{G} \varphi \left(\frac{|u(y^{-1}\cdot x\cdot h)-u(y^{-1}\cdot x)|}{\|h\|_\mathbb{G}^s} \right)\frac{dx}{\|h\|_\mathbb{G}^Q} \right)\rho_{\eps}(y)\,dy\\
        &=\int_\mathbb{G}\varphi\left(\frac{|u(x\cdot h)-u(x)|}{\|h\|_\mathbb{G}^s} \right) \frac{dx}{\|h\|_\mathbb{G}^Q}.
    \end{split}
\end{equation*}
Thus, the thesis follows, by integrating in $\mathbb{G}$ with respect to $h$.
\end{proof}
\begin{lem}\label{2.14}
Let $u\in L^{\varphi}(\B{G})$ and let $\{u_k\}_k$ the sequence of truncated functions of $u$, in the sense of Definition \ref{def_trun}. Then
\[
\Phi_{s,\varphi}(u_k) \leq \frac{\C}{2}\left(\Phi_{s,\varphi}(u) + \left(\frac{2}{k}\right)^{p^-}\frac{QC_b}{(1-s)p^+}\Phi_\varphi(u)+2^{p^+}\frac{QC_b}{sp^-}\Phi_\varphi(u)\right)
\]
for any $k\in\B{N}$ and $0<s<1$.
\end{lem}
\begin{proof}
Let us fix $x,y\in\mathbb{G}$. Then, by the $\Delta_2$-condition and the monotonicity of $\varphi$, we have
\begin{equation*}
    \begin{split}
         \varphi& \left( \frac{|u_k(x)-u_k(y)|}{\|y^{-1}\cdot x\|_\mathbb{G}^s} \right) \\
         &=  \varphi\left( \frac{|\eta_k(x)u(x) -\eta_k(y)u(x) + \eta_k(y)u(x) - \eta_k(y)u(y)|}{\|y^{-1}\cdot x\|_\mathbb{G}^s} \right)\\
         &\leq \frac{\C}{2} \varphi\left( \frac{|u(x)||\eta_k(x) - \eta_k(y)|}{\|y^{-1}\cdot x\|_\mathbb{G}^s} \right) + \frac{\C}{2} \varphi\left( \frac{|\eta_k(y)||u(x)-u(y)|}{\|y^{-1}\cdot x\|_\mathbb{G}^s} \right).
    \end{split}
\end{equation*}
Hence, being $\eta_k\leq 1$ for any $k\in\mathbb{N}$, we get
\begin{equation*}
    \begin{split}
        \Phi_{s,\varphi}(u_k) &= \iint_{\mathbb{G}\times\mathbb{G}}\varphi\left( \frac{|u_k(x)-u_k(y)|}{\|y^{-1}\cdot x\|_\mathbb{G}^s}\right)\frac{dx\,dy}{\|y^{-1}\cdot x\|_\mathbb{G}^Q} \\
        & \leq \frac{\C}{2}\Phi_{s,\varphi}(u)+\frac{\C}{2} \iint_{\mathbb{G}\times\mathbb{G}} \varphi\left( \frac{|u(x)||\eta_k(x)-\eta_k(y)|}{\|y^{-1}\cdot x\|_\mathbb{G}^s} \right) \frac{dx\,dy}{\|y^{-1}\cdot x\|_\mathbb{G}^Q}\\
        & = \frac{\C}{2}\Phi_{s,\varphi}(u)+\frac{\C}{2}\left[ \int_\mathbb{G}\int_{\{\|y^{-1}\cdot x\|_\mathbb{G}<1\}}\varphi\left(\frac{|u(x)||\eta_k(x) - \eta_k(y)|}{\|y^{-1}\cdot x\|_\mathbb{G}^s} \right) \frac{dx\,dy}{\|y^{-1}\cdot x\|_\mathbb{G}^Q}\right. \\
        & \quad+\left.\int_\mathbb{G}\int_{\{\|y^{-1}\cdot x\|_\mathbb{G}\geq 1\}}\varphi\left(\frac{|u(x)||\eta_k(x) - \eta_k(y)|}{\|y^{-1}\cdot x\|_\mathbb{G}^s} \right) \frac{dx\,dy}{\|y^{-1}\cdot x\|_\mathbb{G}^Q}\right].
    \end{split}
\end{equation*}
Since $\|\nabla\eta_k\|_{\mathbb{R}^m}\leq\frac{2}{k}$, then, by \eqref{G1}, assuming without loss of generality $k>2$, and by Proposition \ref{prop.radialfunct}, we have
\begin{equation*}
    \begin{split}
         \int_\mathbb{G}\int_{\{\|y^{-1}\cdot x\|_\mathbb{G}<1\}} &\varphi\left( \frac{|u(x)||\eta_k(x) - \eta_k(y)|}{\|y^{-1}\cdot x\|_\mathbb{G}^s} \right) \frac{dx\,dy}{\|y^{-1}\cdot x\|_\mathbb{G}^Q}\\
         &\leq \int_\mathbb{G}\left(\int_{\{\|y^{-1}\cdot x\|_\mathbb{G}\leq1\}} \varphi\left(\frac{2}{k}\frac{|u(x)|}{\|y^{-1}\cdot x\|_\mathbb{G}^{s-1}}\right)\frac{dy}{\|y^{-1}\cdot x\|_\mathbb{G}^Q}\right)\,dx\\
         &\leq\left(\frac{2}{k}\right)^{p^-} \int_\mathbb{G}\left(\int_{\{\|y^{-1}\cdot x\|_\mathbb{G}\leq1\}} \frac{dy}{\|y^{-1}\cdot x\|_\mathbb{G}^{(s-1)p^++Q}}\right)\varphi(|u(x)|)\,dx\\
         &=\Phi_\varphi(u)\left(\frac{2}{k}\right)^{p^-}QC_b\int_0^1r^{(1-s)p^+-1}\,dr=\left(\frac{2}{k}\right)^{p^-}\frac{QC_b}{(1-s)p^+}\Phi_\varphi(u).
         \end{split}
\end{equation*}
Moreover
\begin{equation*}
    \begin{split}
        \int_\mathbb{G}\int_{\{\|y^{-1}\cdot x\|_\mathbb{G}\geq 1\}} &\varphi\left( \frac{|u(x)||\eta_k(x) - \eta_k(y)|}{\|y^{-1}\cdot x\|_\mathbb{G}^s} \right) \frac{dx\,dy}{\|y^{-1}\cdot x\|_\mathbb{G}^Q}\\
        &\leq\int_\mathbb{G}\int_{\{\|y^{-1}\cdot x\|_\mathbb{G}\geq 1\}} \varphi\left( \frac{2|u(x)|}{\|y^{-1}\cdot x\|_\mathbb{G}^s} \right) \frac{dx\,dy}{\|y^{-1}\cdot x\|_\mathbb{G}^Q}\\
        & \leq 2^{p^+} \int_\mathbb{G}\left(\int_{\{\|y^{-1}\cdot x\|_\mathbb{G}\geq 1\}}\frac{dy}{\|y^{-1}\cdot x\|_\mathbb{G}^{sp^-+Q}}\right)\varphi(|u(x)|)\,dx\\
        & = \Phi_\varphi(u)2^{p^+}QC_b\int_1^\infty r^{-sp^--1}\,dr=2^{p^+}\frac{QC_b}{sp^-}\Phi_\varphi(u).
    \end{split}
\end{equation*}
\end{proof}
\section{A Bougain-Brezis-Mironescu-type Theorem}\label{section3}
In order to prove Theorem \ref{BBM}, we need the two following fundamental Lemmas.
\begin{lem}\label{lemma4.2}
Let $u\in W^{1,\varphi}(\B{G})$. Then, for any $0<s<1$, it holds that
\[
\Phi_{s,\varphi}(u)\leq\frac{QC_b}{p^-}\left( \frac{1}{1-s}\Phi_{\varphi}(\|\nabla_\mathbb{G} u\|_{\mathbb{R}^m}) + \frac{\C}{s}\Phi_{\varphi}(u)\right),
\]
where $\C$ is the $\Delta_2$-constant given in \eqref{doubling} and $p^-$ is given in \eqref{L}.
\end{lem}
\begin{proof}
Let $u\in C_c^2(\B{G})$ and $h=y^{-1}\cdot x$. It follows that
\begin{equation*}
    \begin{split}
        \Phi_{s,\varphi}(u) &= 
        \int_\mathbb{G}\left(\int_{\{\|h\|_\mathbb{G}<1\}}\varphi \left( \frac{|u(x\cdot h)-u(x)|}{\|h\|_\mathbb{G}^s}\right)\frac{dh}{\|h\|_\mathbb{G}^Q}\right)\,dx\\
        &\quad+\int_\mathbb{G}\left(\int_{\{\|h\|_\mathbb{G}\geq 1\}}\varphi \left( \frac{|u(x\cdot h)-u(x)|}{\|h\|_\mathbb{G}^s}\right)\frac{dh}{\|h\|_\mathbb{G}^Q}\right)\,dx= I_1 + I_2.
    \end{split}
\end{equation*}
Let us start from $I_1$. Observe that $u\in C_c^2(\B{G})$ implies that $u$ is Pansu differentiable (see also \cite[Section 2]{B}). If we define the auxiliary function $\xi(t) \coloneqq u(x\cdot\delta_t h)$, then we can write
\begin{equation*}
u(x\cdot h)-u(x) = \xi(1) - \xi(0) = \int_0^1 \frac{d}{dt}\xi(t)\,dt = \int_0^1 \nabla_\mathbb{G} u(x\cdot\delta_t h)\cdot h'\,dt.
\end{equation*}
Therefore, by the monotonicity and the convexity of $\varphi$, we get
\begin{equation}\label{eq.I1lemma4.2}
    \begin{split}
        \varphi\left( \frac{|u(x\cdot h)-u(x)|}{\|h\|_\mathbb{G}^s} \right) &\leq  \varphi\left(\int_0^1 \frac{|\nabla_\mathbb{G} u(x\cdot\delta_t h)\cdot h'|}{\|h\|_\mathbb{G}^s}\,dt\right)\\
        & \leq \int_0^1 \varphi \left( \frac{|\nabla_\mathbb{G} u(x\cdot\delta_t h)\cdot h'|}{\|h\|_\mathbb{G}^s} \right)\,dt\\
        & \leq \int_0^1 \varphi ( \|\nabla_\mathbb{G} u(x\cdot\delta_t h)\|_{\mathbb{R}^m}\|h\|_\mathbb{G}^{1-s})\,dt.
    \end{split}
\end{equation}
Thus, by \eqref{G1} and Proposition \ref{prop.radialfunct}
\begin{align*}
    I_1&\leq\int_\mathbb{G}\left(\int_{\{\|h\|_\mathbb{G}<1\}}\left(\int_0^1 \varphi ( \|\nabla_\mathbb{G} u(x\cdot\delta_t h)\|_{\mathbb{R}^m}\|h\|_\mathbb{G}^{1-s})\,dt\right)\,\frac{dh}{\|h\|_\mathbb{G}^Q}\right)\,dx\\
    &\leq\int_\mathbb{G}\left(\int_{\{\|h\|_\mathbb{G}<1\}}\left(\int_0^1 \varphi ( \|\nabla_\mathbb{G} u(x\cdot\delta_t h)\|_{\mathbb{R}^m})\,dt\right)\,\frac{\|h\|_\mathbb{G}^{(1-s)p^-}}{\|h\|_\mathbb{G}^Q}\,dh\right)\,dx\\
    &=\int_{\{\|h\|_\mathbb{G}<1\}}\|h\|_\mathbb{G}^{(1-s)p^--Q}\,dh\int_\mathbb{G}\varphi( \|\nabla_\mathbb{G} u(x)\|_{\mathbb{R}^m})\,dx\\
    &=QC_b\int_0^1r^{(1-s)p^--1}\,dr\,\Phi_\varphi(\|\nabla u\|_{\mathbb{R}^m})=\frac{QC_b}{(1-s)p^-}\Phi_\varphi(\|\nabla u\|_{\mathbb{R}^m}).
\end{align*}
Moreover, by \eqref{G1}, \eqref{G2}, Proposition \ref{prop.radialfunct}, the monotonicity of $\varphi$ and by a change of variables, we have
\begin{align*}
    I_2&\leq\int_\mathbb{G}\left(\int_{\{\|h\|_\mathbb{G}\geq 1\}}\varphi\left(|u(x\cdot h)|+|u(x)|\right)\frac{dh}{\|h\|_\mathbb{G}^{sp^-+Q}}\right)\,dx\\
    &\leq\frac{\C}{2}\int_\mathbb{G}\left(\int_{\{\|h\|_\mathbb{G}\geq 1\}}\varphi\left(|u(x\cdot h)|\right)\frac{dh}{\|h\|_\mathbb{G}^{sp^-+Q}}\right)\,dx\\
    &\quad+\frac{\C}{2}\int_\mathbb{G}\left(\int_{\{\|h\|_\mathbb{G}\geq 1\}}\varphi\left(|u(x)|\right)\frac{dh}{\|h\|_\mathbb{G}^{sp^-+Q}}\right)\,dx\\
    &=\C\int_\mathbb{G}\left(\int_{\{\|h\|_\mathbb{G}\geq 1\}}\varphi\left(|u(x)|\right)\frac{dh}{\|h\|_\mathbb{G}^{sp^-+Q}}\right)\,dx\\
    &=\C\int_{\{\|h\|_\mathbb{G}\geq 1\}}\frac{dh}{\|h\|_\mathbb{G}^{sp^-+Q}}\int_\mathbb{G}\varphi\left(|u(x)|\right)\,dx\\
    &=\C QC_b\int_1^{+\infty}r^{-sp^--1}\,dr\,\Phi_\varphi(u)=\C\frac{QC_b}{sp^-}\Phi_\varphi(u).
\end{align*}
Finally, for any $u\in W^{1,\varphi}(\B{G})$, let $\{ u_k \}_k \subset  C_c^2(\B{G})$ be convergent to $u$ in $W^{1,\varphi}(\B{G})$. Thus, by the Fatou's Lemma and the continuity of $\varphi$, we get
\begin{equation*}
    \begin{split}
        \Phi_{s,\varphi}(u) &\leq \liminf_{k\to\infty} \Phi_{s,\varphi}(u_k) \leq \lim_{k\to\infty}\left[\frac{QC_b}{p^-}\left( \frac{1}{1-s}\Phi_{\varphi}(\|\nabla_\mathbb{G} u_k\|_{\mathbb{R}^m}) + \frac{\C}{s}\Phi_{\varphi}(u_k)\right) \right]\\
        & =\frac{QC_b}{p^-}\left( \frac{1}{1-s}\Phi_{\varphi}(\|\nabla_\mathbb{G} u\|_{\mathbb{R}^m}) + \frac{\C}{s}\Phi_{\varphi}(u)\right)
    \end{split}
\end{equation*}
as desired.
\end{proof}
\begin{lem}\label{lemma4.3}
Let $\varphi$ be an Orlicz function such that $\tilde{\varphi}$ exists and let $u\in C^2_c(\B{G})$. Then, for every fixed $x\in\B{G}$, we have that
\begin{equation*}
    \lim_{s \uparrow 1}(1-s)\int_{\B{G}} \varphi\left( \frac{|u(x)-u(y)|}{\|y^{-1}\cdot x\|_\mathbb{G}^s} \right) \frac{dy}{\|y^{-1}\cdot x\|_\mathbb{G}^Q} = \tilde{\varphi}(\|\nabla_\mathbb{G} u\|_{\mathbb{R}^m}).
\end{equation*}
\end{lem}
\begin{proof}
For each fixed $x\in\B{G}$ we have
\begin{equation*}
    \begin{split}
        \int_{\B{G}}\varphi\left( \frac{|u(x)-u(y)|}{\|y^{-1}\cdot x\|_\mathbb{G}^s} \right)\frac{dy}{\|y^{-1}\cdot x\|_\mathbb{G}^Q} = & \int_{\{\|y^{-1}\cdot x\|_\mathbb{G}<1\}} \varphi\left( \frac{|u(x)-u(y)|}{\|y^{-1}\cdot x\|_\mathbb{G}^s} \right)\frac{dy}{\|y^{-1}\cdot x\|_\mathbb{G}^Q}\\
       & + \int_{\{\|y^{-1}\cdot x\|_\mathbb{G}\geq 1\}} \varphi\left(\frac{|u(x)-u(y)|}{\|y^{-1}\cdot x\|_\mathbb{G}^s}\right)\frac{dy}{\|y^{-1}\cdot x\|_\mathbb{G}^Q}\\
       & = I_1 + I_2.
    \end{split}
\end{equation*}

Let us first notice that
\begin{equation*}
\lim_{s\uparrow 1}(1-s)I_2 = 0.
\end{equation*}
In facty, by \eqref{G1} and Proposition \ref{prop.radialfunct}, we have
\begin{align*} 
    \int_{\{\|y^{-1}\cdot x\|_\mathbb{G}\geq 1\}} \varphi\left(\frac{|u(x)-u(y)|}{\|y^{-1}\cdot x\|_\mathbb{G}^s}\right)\frac{dy}{\|y^{-1}\cdot x\|_\mathbb{G}^Q}&\leq\varphi(2\|u\|_\infty)\int_{\{\|y^{-1}\cdot x\|_\mathbb{G}\geq 1\}}\frac{dy}{\|y^{-1}\cdot x\|_\mathbb{G}^{sp^- +Q}}\\
    &=\varphi(2\|u\|_\infty)QC_b\int_1^{+\infty}r^{-sp^--1}\,dr\\
    &=\frac{QC_b}{sp^-}\varphi(2\|u\|_\infty).
\end{align*}

Now, by the local Lipschitzianity of $\varphi$, for any $x,y\in\B{G}$ such that $x\neq y$, we have
\begin{equation*}
\begin{split}
\left| \varphi\left( \frac{|u(x)-u(y)|}{\|h\|_\mathbb{G}^s} \right)  - \varphi\left( \frac{|\nabla_\mathbb{G} u(x)\cdot h'|}{\|h\|_\mathbb{G}^s} \right) \right| & \leq L \frac{|u(x)-u(y) - \nabla_\mathbb{G} u(x) \cdot h'|}{\|h\|_\mathbb{G}^s}\\
& \leq C\|h\|_\mathbb{G}^{2-s}
\end{split}
\end{equation*}
where $L$ is the Lipschitz constant of $\varphi$ in the interval $[0,\|\nabla_\mathbb{G}u\|_\infty]$, $C$ is a constant depending on the $C^2$-norm of $u$ and $h=y^{-1}\cdot x$. (The last inequality follows from standard results about the Taylor polynomial that can be found, for instance, in \cite[Chapter 20]{BLU}).

Moreover, by Proposition \ref{prop.radialfunct}, it follows that
\[
\int_{\{\|h\|_\mathbb{G}<1\}}\|h\|_\mathbb{G}^{2-s}\frac{dy}{\|h\|_\mathbb{G}^Q} = QC_b\int_0^1 r^{1-s}\,dr = \frac{QC_b}{2-s}.
\]
Thus
\begin{equation*}
    \begin{split}
        &\lim_{s\uparrow 1}(1-s)\int_{\{\|y^{-1}\cdot x\|_\mathbb{G}<1\}}\varphi\left( \frac{|u(x)-u(y)|}{\|y^{-1}\cdot x\|_\mathbb{G}^s} \right)\frac{dy}{\|y^{-1}\cdot x\|_\mathbb{G}^Q}\\
        & = \lim_{s\uparrow 1}(1-s)\int_{\{\|h\|_\mathbb{G}<1\}}\varphi\left( \frac{|\nabla_\mathbb{G} u(x)\cdot h'|}{\|h\|_\mathbb{G}^s} \right)\frac{dy}{\|h\|_\mathbb{G}^Q}.
    \end{split}
\end{equation*}

Finally, by Proposition \ref{prop.fs} and the invariance of $\|\cdot\|_\mathbb{G}$ under horizontal rotations, we have
\begin{equation*}
    \begin{split}
        \int_{\{\|h\|_\mathbb{G}<1\}}\varphi\left( \frac{|\nabla_\mathbb{G} u(x)\cdot h'|}{\|h\|_\mathbb{G}^s} \right)\frac{dy}{\|h\|_\mathbb{G}^Q} &= \int_0^1\left(\int_S \varphi\left(\frac{|\nabla_\mathbb{G} u(x)\cdot \delta_r z'|}{\|\delta_r z\|_\mathbb{G}^s}\right)\frac{d\sigma(z)}{\|\delta_r z\|_\mathbb{G}^Q}\right)r^{Q-1}\,dr\\
        &=\int_0^1\left(\int_S \varphi\left(\frac{|\nabla_\mathbb{G} u(x)\cdot z'|}{r^s\|z\|_\mathbb{G}^s}r\right)\frac{d\sigma(z)}{\|z\|_\mathbb{G}^Q}\right)\frac{r^{Q-1}}{r^Q}\,dr\\
        &=\int_0^1\left(\int_S \varphi(\|\nabla_\mathbb{G} u(x)\|_{\mathbb{R}^m}\|z'\|_{\mathbb{R}^m}r^{1-s}) d\sigma(z)\right)\frac{dr}{r},
    \end{split}
\end{equation*}
i.e.,
\begin{equation*}\label{eq.I1}
\lim_{s\uparrow 1}(1-s)I_1 = \tilde{\varphi}(\|\nabla_\mathbb{G} u(x)\|_{\mathbb{R}^m}).
\end{equation*}
\end{proof}
Finally, we are ready to prove Theorem \ref{BBM}.
\begin{proof}[Proof of Theorem \ref{BBM}]
We divide the proof of the Theorem in three steps. 
\medskip

\underline{\textit{Step 1:}} Let us prove the Theorem for any function $u\in C^2_c({\B{G}})$ whose support is contained in $B(0,R)$. Let
\[
F_s(x)\coloneqq \int_\mathbb{G}\varphi\left( \frac{|u(x)-u(y)|}{\|y^{-1}\cdot x\|_\mathbb{G}^s} \right) \frac{dy}{\|y^{-1}\cdot x\|_\mathbb{G}^Q}.
\]
In virtu of Lemma \ref{lemma4.3}, in order to apply the Dominated Convergence Theorem, it is enough to show the existence of an integrable function in $\mathbb{G}$ that dominates the sequence $\{(1-s)F_s\}_{s\in(0,1)}$.

Let us fix $R>1$. For any $x\in\mathbb{G}$ such that $\|x\|_\mathbb{G}<2R$, we have
\begin{align*}
    F_s(x)&=\int_{\{\|y^{-1}\cdot x\|_\mathbb{G}< 1\}} \varphi\left(\frac{|u(x)-u(y)|}{\|y^{-1}\cdot x\|_\mathbb{G}^s}\right)\frac{dy}{\|y^{-1}\cdot x\|_\mathbb{G}^Q}\\
    &\quad+\int_{\{\|y^{-1}\cdot x\|_\mathbb{G}\geq1\}}\varphi\left(\frac{|u(x)-u(y)|}{\|y^{-1}\cdot x\|_\mathbb{G}^s}\right)\frac{dy}{\|y^{-1}\cdot x\|_\mathbb{G}^Q}= I_1 + I_2.
\end{align*}
By \eqref{eq.I1lemma4.2}, \eqref{G1}, Proposition \ref{prop.radialfunct} and the monotonicity of $\varphi$, called $h=y^{-1}\cdot x$, we have
\begin{equation*}
    \begin{split}
        I_1&\leq\int_{\{\|h\|_\mathbb{G}<1\}}\left(\int_0^1\varphi ( \|\nabla_\mathbb{G} u(x\cdot\delta_t h)\|_{\mathbb{R}^m}\|h\|_\mathbb{G}^{1-s})\,dt\right)\frac{dh}{\|h\|_\mathbb{G}^Q}\\
        &\leq\int_{\{\|h\|_\mathbb{G}<1\}}\left(\int_0^1\varphi(\|\nabla_\mathbb{G} u(x\cdot\delta_t h)\|_{\mathbb{R}^m})\,dt\right)\frac{\|h\|_\mathbb{G}^{(1-s)p^-}}{\|h\|_\mathbb{G}^Q}\,dh\\
        &\leq\varphi(||\nabla_\mathbb{G} u||_{\infty})QC_b\int_0^1 r^{(1-s)p^--1}\,dr=\frac{QC_b}{(1-s)p^-}\varphi(||\nabla_\mathbb{G} u||_{\infty}).
    \end{split}
\end{equation*}
Moreover, by \eqref{G1} and Proposition \ref{prop.radialfunct}
\begin{equation*}
    \begin{split}
        I_2&\leq\int_{\{\|y^{-1}\cdot x\|_\mathbb{G}\geq1\}}\varphi\left(|u(x)|+|u(y)|\right)\frac{dy}{\|y^{-1}\cdot x\|_\mathbb{G}^{sp^-+Q}}\\
        &\leq\varphi(2||u||_{\infty})QC_b\int_1^{\infty}r^{-sp^--1}\,dr=\frac{QC_b}{sp^-}\varphi(2||u||_{\infty}).
    \end{split}
\end{equation*}
Thus
\begin{equation}\label{chi1}
    F_s(x)\leq\frac{QC_b}{(1-s)p^-}\varphi(||\nabla_\mathbb{G} u||_{\infty})+\frac{QC_b}{sp^-}\varphi(2||u||_{\infty})
\end{equation}
for every $\|x\|_\mathbb{G}<2R$.

Let $\|x\|_\mathbb{G}\geq 2R$. Since the support of $u$ is contained in $B(0,R)$, then $u(z)=0$ for any $\|z\|_\mathbb{G}>R$. Thus
\[
F_s(x) = \int_{\{\|y\|_\mathbb{G}\leq R\}}\varphi\left( \frac{|u(y)|}{\|y^{-1}\cdot x\|_\mathbb{G}^s} \right) \frac{dy}{\|y^{-1}\cdot x\|_\mathbb{G}^Q}.
\]
Now, since $\|y^{-1}\cdot x\|_\mathbb{G}\geq\|x\|_\mathbb{G}-\|y\|_\mathbb{G}\geq\|x\|_\mathbb{G}-R\geq\frac{1}{2}\|x\|_\mathbb{G}$, then by the monotonicity of $\varphi$, the $\Delta_2$-condition and \eqref{G1}, we get
\begin{equation*}
    \begin{split}
        F_s(x)&\leq\int_{\{\|y\|_\mathbb{G}\leq R\}}\varphi\left( \frac{|u(y)|}{(\frac{1}{2}\|x\|_\mathbb{G})^s} \right)\frac{dy}{(\frac{1}{2}\|x\|_\mathbb{G})^Q}\\
        &\leq2^Q\int_{\{\|y\|_\mathbb{G}\leq R\}}\varphi\left(2^s|u(y)| \right)\frac{dy}{\|x\|_\mathbb{G}^{sp^-+Q}}\\
        &\leq\C\frac{2^Q}{\|x\|_\mathbb{G}^{sp^-+Q}}\int_{\{\|y\|_\mathbb{G}\leq R\}}\varphi\left(|u(y)| \right)\,dy\\
        &\leq\C\frac{2^Q}{\|x\|_\mathbb{G}^{\frac{1}{2}p^-+Q}}\int_{\{\|y\|_\mathbb{G}\leq R\}}\varphi\left(|u(y)|\right)\,dy,
    \end{split}
\end{equation*}
that is
\begin{equation}\label{chi2}
    F_s(x)\leq\frac{K}{\|x\|_\mathbb{G}^{\frac{1}{2}p^-+Q}}
\end{equation}
for every $\|x\|_\mathbb{G}\geq 2R$, where we assumed $s\geq 1/2$. Here $K$ is a constant independent of $s$.

Therefore, by \eqref{chi1} and \eqref{chi2}, we have
\begin{align*}
    F_s(x)&\leq\left(\frac{QC_b}{(1-s)p^-}\varphi(||\nabla_\mathbb{G} u||_{\infty})+\frac{QC_b}{sp^-}\varphi(2||u||_{\infty})\right)\chi_{B(0,2R)}(x)\\
    &\quad+\frac{K}{\|x\|_\mathbb{G}^{\frac{1}{2}p^-+Q}}\chi_{\mathbb{G}\setminus B(0,2R)}(x)=:H(x),
\end{align*}
i.e.,
\begin{align*}
    (1-s)F_s(x)\leq(1-s)H(x)\in L^1(\B{G}),
\end{align*}
as desired.

\underline{\textit{Step 2:}} Let $u\in W^{1,\varphi}(\B{G})$. Then, by Theorem \ref{2.10}, there exists $\{u_{k}\}_{k\in\mathbb{N}}\subset C^2_c({\B{G}})$ such that $u_k\to u$ in $W^{1,\varphi}(\B{G})$. Let us show that
\[
\lim_{s\uparrow 1}(1-s)\Phi_{s,\varphi}(u) = \Phi_{\tilde{\varphi}}(\|\nabla_\mathbb{G} u\|_{\mathbb{R}^m}).
\]
Being
\begin{align*}
    |(1-s)\Phi_{s,\varphi}(u)-\Phi_{\tilde{\varphi}}(\|\nabla_\mathbb{G} u\|_{\mathbb{R}^m})|&\leq |(1-s)\Phi_{s,\varphi}(u)-(1-s)\Phi_{s,\varphi}(u_k)|\\&
    +|(1-s)\Phi_{s,\varphi}(u_k)-\Phi_{\tilde{\varphi}}(\|\nabla_\mathbb{G} u_k\|_{\mathbb{R}^m})|\\&
    +|\Phi_{\tilde{\varphi}}(\|\nabla_\mathbb{G} u_k \|_{\mathbb{R}^m})-\Phi_{\tilde{\varphi}}(\|\nabla_\mathbb{G} u\|_{\mathbb{R}^m})|
\end{align*}
then, by Lemma \ref{lemma4.3}, it only remains to show that for any $s\in(0,1)$ and for any $\varepsilon>0$ there exists $\overline{k}\in\mathbb{N}$ such that
\begin{align*}
    (1-s)|\Phi_{s,\varphi}(u)-\Phi_{s,\varphi}(u_k)|+|\Phi_{\tilde{\varphi}}(\|\nabla_\mathbb{G} u_k\|_{\mathbb{R}^m})-\Phi_{\tilde{\varphi}}(\|\nabla_\mathbb{G} u\|_{\mathbb{R}^m})|<\varepsilon
\end{align*}
$\forall k\geq\overline{k}$.

Fixed $\varepsilon>0$, by Theorem \ref{2.10}, there exists $k_0\in\mathbb{N}$ such that
\begin{align*}\label{grad}
    |\Phi_{\tilde{\varphi}}(\|\nabla_\mathbb{G} u_k\|_{\mathbb{R}^m})-\Phi_{\tilde{\varphi}}(\|\nabla_\mathbb{G} u\|_{\mathbb{R}^m})|<\dfrac{\varepsilon}{2}
\end{align*}
for any $k\geq k_0$. Moreover, by Lemma \ref{lem.2.6}, for any $\delta>0$ there exists $C_\delta>0$ such that
\[
\varphi(s+t) \leq C_{\delta}\varphi(s) + (1+\delta)^{p^+}\varphi(t)\ \text{for any }s,t\geq 0.
\]

Moreover, there exists $\overline{\delta}>0$ such that $(1+\delta)^{p^+}\leq 1+\overline{\delta}$. Therefore
\begin{align*}
    |\Phi_{s,\varphi}(u)-\Phi_{s,\varphi}(u_k)|&\leq\iint_{\mathbb{G}\times\mathbb{G}}\bigg|\varphi\bigg(\frac{|(u-u_k)(x)-(u-u_k)(y)|}{\|y^{-1}\cdot x\|_\mathbb{G}^s}\\
    &\quad+\frac{|u_k(x)-u_k(y)|}{\|y^{-1}\cdot x\|_\mathbb{G}^s}\bigg)-\varphi\left(\frac{|u_k(x)-u_k(y)|}{\|y^{-1}\cdot x\|_\mathbb{G}^s}\right)\bigg|\frac{dx\,dy}{\|y^{-1}\cdot x\|_\mathbb{G}^Q}\\
    &\leq C_\delta\Phi_{s,\varphi}(u-u_k)+\overline{\delta}\Phi_{s,\varphi}(u_k).
\end{align*}
Taking into account Lemma \ref{lemma4.2}, and being $u_k\to u$ in $W^{1,\varphi}(\B{G})$, then there exists $k_1\in\mathbb{N}$ such that
\begin{align*}
    \Phi_{s,\varphi}(u-u_k)\leq\frac{QC_b}{p^-}\left( \frac{1}{1-s}\Phi_{\varphi}(\|\nabla_\mathbb{G} (u-u_k)\|_{\mathbb{R}^m}) + \frac{\C}{s}\Phi_{\varphi}(u-u_k)\right)
\end{align*}
for any $k\geq k_1$, that is, 
\begin{align*}
    (1-s)\Phi_{s,\varphi}(u-u_k)\leq\dfrac{\varepsilon}{4C_\delta}.
\end{align*}
Moreover, still by Lemma \ref{lemma4.2}, there exists a positive constant $M$ such that
\begin{align*}
    \Phi_{s,\varphi}(u_k)\leq\frac{QC_b}{p^-}\left( \frac{1}{1-s}\Phi_{\varphi}(\|\nabla_\mathbb{G} (u_k)\|_{\mathbb{R}^m}) + \frac{\C}{s}\Phi_{\varphi}(u_k)\right)\leq M.
\end{align*}

Then, taking $\overline{\delta}\leq\dfrac{\varepsilon}{4M(1-s)}$, we get the thesis, as $s\to 1$ for any $k\geq\max\{k_0,k_1\}$.

\underline{\textit{Step 3:}}
In order to conclude the proof of the Theorem, let us prove the result for any $u\in L^\varphi(\B{G})$.

Let us fix $u\in L^\varphi(\B{G})$, $k\in\mathbb{N}$ and $\varepsilon>0$ and let us define
\begin{align*}
    u_{k,\varepsilon}\coloneqq(\eta_k u)*\rho_\epsilon\in C^\infty_c(\B{G})
\end{align*}
where $\{\rho_\varepsilon\}_\varepsilon$ is a sequence of mollifiers and $\{\eta_k\}_k$ is a truncated sequence, in the sense of Definition \ref{def_moll} and Definition \ref{def_trun}.

Then, by Lemma \ref{2.13} and Lemma \ref{2.14}, there exists a positive constant $N$, independent of $k$ and $\varepsilon$, such that
\begin{align*}
    \liminf_{s\uparrow 1}(1-s)\Phi_{s,\varphi}(u_{k,\varepsilon})<N.
\end{align*}
Therefore, by \textit{Step 1}
\begin{align*}
    \Phi_{\tilde{\varphi}}(\|\nabla_\mathbb{G} u_{k,\varepsilon}\|_{\mathbb{R}^m})|<\infty,
\end{align*}
i.e., the sequence $\{u_{k,\varepsilon}\}_{k,\varepsilon}$ is bounded in $W^{1.\tilde{\varphi}}(\B{G})$ and then, in virtue of Proposition \ref{2.16}, $\{u_{k,\varepsilon}\}_{k,\varepsilon}$ is bounded in $W^{1.\varphi}(\B{G})$.

Thus, by the reflexivity of the space $W^{1,\varphi}(\B{G})$, there exists $\tilde{u}\in W^{1,\varphi}(\B{G})$ such that, up to subsequence,
\begin{align*}
    u_{k,\varepsilon}\rightharpoonup\tilde{u}\ \text{weakly in }W^{1,\varphi}(\B{G})
\end{align*}
as $k\uparrow\infty$ and $\varepsilon\downarrow 0$. Thus, being $u_{k,\varepsilon}\to u$ in $L^\varphi(\B{G})$, it follows that $\tilde{u}=u$ in $W^{1,\varphi}(\B{G})$. Finally, the thesis holds by \textit{Step 2}.
\end{proof}
\section{Maz'ya-Shaposhnikova-type Theorems}\label{section4}
The last part of the work consists on some generalizations of the well-known paper of Maz'ya-Shaposhnikova \cite{MS}. 
\begin{theorem}[Liminf estimate]\label{thliminf}
For any $u\in W^{s,\varphi}(\mathbb{G})$ it holds that
\begin{equation}\label{liminf}
    \liminf_{s\downarrow 0}s\Phi_{s,\varphi}(u)\geq \dfrac{4}{\C}\dfrac{QC_b}{p^+}\Phi_\varphi(u),
\end{equation}
where $2<\C\leq 2^{p^+}$ is the $\Delta_2$-constant of Definition \ref{delta2}.
\end{theorem}

\begin{proof}

Observe that we can express $\Phi_{s,\varphi}(u)$ as the sum of the terms
\begin{align*}
(i)&:= \int_{\mathbb{G}} \left(\int_{\{\|y^{-1}\cdot x\|_\mathbb{G}>2\|x\|_\mathbb{G}\}}  \varphi\left(\frac{|u(x)-u(y)|}{\|y^{-1}\cdot x\|_\mathbb{G}^s}\right) \,\frac{dy}{\|y^{-1}\cdot x\|_\mathbb{G}^Q}\right)\,dx\\
(ii)&:=  \int_{\mathbb{G}} \left(\int_{\{\|y^{-1}\cdot x\|_\mathbb{G}\leq 2\|x\|_\mathbb{G}\}}  \varphi\left(\frac{|u(x)-u(y)|}{\|y^{-1}\cdot x\|_\mathbb{G}^s}\right) \,\frac{dy}{\|y^{-1}\cdot x\|_\mathbb{G}^Q}\right)\,dx.
\end{align*}
We claim that
\begin{equation}\label{0}
 \Phi_{s,\varphi}(u) \geq 2\,(i).
\end{equation}
To prove \eqref{0}, first notice that, by changes of variables and the Fubini's Theorem
\begin{align}\label{sd}
\begin{split}
     (i)
     &=\int_\mathbb{G}\left(\int_{\{\|z\|_\mathbb{G}>2\|x\|_\mathbb{G}\}}\varphi\left(\frac{|u(x)-u(x\cdot z^{-1})|}{\|z\|_\mathbb{G}^s}\right)\,\frac{dz}{\|z\|_\mathbb{G}^Q}\right)\,dx\\
     &=\int_\mathbb{G}\left(\int_{\{\|x\|_\mathbb{G}<\frac{1}{2}\|z\|_\mathbb{G}\}}\varphi\left(\frac{|u(x)-u(x\cdot z^{-1})|}{\|z\|_\mathbb{G}^s}\right)\,\frac{dx}{\|z\|_\mathbb{G}^Q}\right)\,dz\\
     &=\int_\mathbb{G}\left(\int_{\{\|y^{-1}\cdot x\|_\mathbb{G}>2\|x\|_\mathbb{G}\}}\varphi\left(\frac{|u(x)-u(y)|}{\|y^{-1}\cdot x\|_\mathbb{G}^s}\right)\,\frac{dx}{\|y^{-1}\cdot x\|_\mathbb{G}^Q}\right)\,dy.
     \end{split}
\end{align}
Moreover, since the triangular inequality gives the inclusion
\begin{align*}
    \{(x,y)\in\mathbb{G}\times\mathbb{G}\colon \|y^{-1}\cdot x\|_\mathbb{G}>2\|x\|_\mathbb{G}\}\subset\{(x,y)\in\mathbb{G}\times\mathbb{G}\colon\|y^{-1}\cdot x\|_\mathbb{G}\leq 2\|y\|_\mathbb{G}\},
\end{align*}
then, by \eqref{sd}, we get
\begin{align*} 
       (ii)
        &=\int_{\mathbb{G}}\left(\int_{\{\|y^{-1}\cdot x\|_\mathbb{G}\leq2\|y\|_\mathbb{G}\}}\varphi\left(\frac{|u(x)-u(y)|}{\|y^{-1}\cdot x\|_\mathbb{G}^s}\right)\,\frac{dx}{\|y^{-1}\cdot x\|_\mathbb{G}^Q}\right)\,dy\\
        &\geq \int_{\mathbb{G}}\left(\int_{\{\|y^{-1}\cdot x\|_\mathbb{G}>2\|x\|_\mathbb{G}\}}\varphi\left(\frac{|u(x)-u(y)|}{\|y^{-1}\cdot x\|_\mathbb{G}^s}\right)\,\frac{dx}{\|y^{-1}\cdot x\|_\mathbb{G}^Q}\right)\,dy\\
        &=\int_{\mathbb{G}}\left(\int_{\{\|y^{-1}\cdot x\|_\mathbb{G}>2\|x\|_\mathbb{G}\}}\varphi\left(\frac{|u(x)-u(y)|}{\|y^{-1}\cdot x\|_\mathbb{G}^s}\right)\,\frac{dy}{\|y^{-1}\cdot x\|_\mathbb{G}^Q}\right)\,dx
\end{align*}
and hence \eqref{0} follows.

\medskip
Now, observe that, by the $\Delta_2$-condition and \eqref{0}, we get
\begin{equation} \label{00}
    \int_{\mathbb{G}} \left(\int_{\{\|y^{-1}\cdot x\|_\mathbb{G}>2\|x\|_\mathbb{G}\}}\varphi\left(\frac{|u(x)|}{\|y^{-1}\cdot x\|_\mathbb{G}^s}\right)\,\frac{dy}{\|y^{-1}\cdot x\|_\mathbb{G}^Q}\right)\,dx \leq \frac{\C}{2} (i)+
    \frac{\C}{2} (iii)
\end{equation}    
where
$$
(iii):=\int_{\mathbb{G}}\left(\int_{\{\|y^{-1}\cdot x\|_\mathbb{G}>2\|x\|_\mathbb{G}\}}\varphi\left(\frac{|u(y)|}{\|y^{-1}\cdot x\|_\mathbb{G}^s}\right)\,\frac{dy}{\|y^{-1}\cdot x\|_\mathbb{G}^Q}\right)\,dx.
$$
Moreover,  by the triangular inequality, $\|y^{-1}\cdot x\|_\mathbb{G}>2\|x\|_\mathbb{G}$ implies that
\begin{itemize}
    \item $\|y^{-1}\cdot x\|_\mathbb{G}\leq\|x\|_\mathbb{G}+\|y\|_\mathbb{G}<\frac{1}{2}\|y^{-1}\cdot x\|_\mathbb{G}+\|y\|_\mathbb{G}$,
    \item $\|y^{-1}\cdot x\|_\mathbb{G}\geq\|y\|_\mathbb{G}-\|x\|_\mathbb{G}>\|y\|_\mathbb{G}-\frac{1}{2}\|y^{-1}\cdot x\|_\mathbb{G}$,
\end{itemize}
that is
\begin{equation*}
    \{\|y^{-1}\cdot x\|_\mathbb{G}>2\|x\|_\mathbb{G}\}\subset\left\{\frac{2}{3}\|y\|_\mathbb{G}<\|y^{-1}\cdot x\|_\mathbb{G}<2\|y\|_\mathbb{G}\right\}.
\end{equation*}
Therefore, from \eqref{G1} and  Proposition \ref{prop.radialfunct}, we find that  
\begin{align} \label{000}
\begin{split}
(iii) &\leq
    \int_{\mathbb{G}}\left(\int_{\{\frac{2}{3}\|y\|_\mathbb{G}<\|y^{-1}\cdot x\|_\mathbb{G}<2\|y\|_\mathbb{G}\}}\frac{dx}{\|y^{-1}\cdot x\|_\mathbb{G}^{s\tilde{p}+Q}}\right)\varphi(|u(y)|)\,dy\\
    &=QC_b\int_{\mathbb{G}}\left(\int_{\frac{2}{3}\|y\|_\mathbb{G}}^{2\|y\|_\mathbb{G}}r^{-s\tilde{p}-1}\,dr\right)\varphi(|u(y)|)\,dy\\
    &=\frac{1}{s} \frac{QC_b}{2^{s\tilde{p}}\tilde{p}}(3^{s\tilde{p}}-1)\int_{\mathbb{G}}\frac{\varphi(|u(y)|)}{\|y\|_\mathbb{G}^{s\tilde{p}}}\,dy.
\end{split}
\end{align}
Furthermore, by \eqref{G1} and Proposition \ref{prop.radialfunct}, it follows that
\begin{align} \label{1}
\begin{split}
    \int_{\mathbb{G}}&\left(\int_{\{\|y^{-1}\cdot x\|_\mathbb{G}>2\|x\|_\mathbb{G}\}}\varphi\left(\frac{|u(x)|}{\|y^{-1}\cdot x\|_\mathbb{G}^s}\right)\,\frac{dy}{\|y^{-1}\cdot x\|_\mathbb{G}^Q}\right)\,dx\\
    &\geq\int_{\mathbb{G}}\left(\int_{\{\|y^{-1}\cdot x\|_\mathbb{G}>2\|x\|_\mathbb{G}\}}\frac{dy}{\|y^{-1}\cdot x\|_\mathbb{G}^{s\overline{p}+Q}}\right)\,\varphi(|u(x)|)\,dx\\
    &=\dfrac{QC_b}{2^{s\overline{p}}s\overline{p}}\int_\mathbb{G}\frac{\varphi(|u(x)|)}{\|x\|_\mathbb{G}^{s\overline{p}}}\,dx\geq\dfrac{QC_b}{2^{s\overline{p}}sp^+}\int_\mathbb{G}\frac{\varphi(|u(x)|)}{\|x\|_\mathbb{G}^{s\overline{p}}}\,dx.
\end{split}
\end{align}

Hence, gathering  \eqref{0}, \eqref{00}, \eqref{000} and \eqref{1} we obtain that
$$
    s\Phi_{s,\varphi}(u)\geq\dfrac{4}{\C}\dfrac{QC_b}{2^{s\overline{p}}p^+}\int_\mathbb{G}\frac{\varphi(|u(x)|)}{\|x\|_\mathbb{G}^{s\overline{p}}}\,dx-\frac{2QC_b}{2^{s\tilde{p}}\tilde{p}}(3^{s\tilde{p}}-1)\int_{\mathbb{G}}\frac{\varphi(|u(x)|)}{\|x\|_\mathbb{G}^{s\tilde{p}}}\,dx.
$$
Finally, by the Fatou's Lemma, we  get
\begin{align*}
    \liminf_{s\downarrow 0}s\Phi_{s,\varphi}(u)&\geq\liminf_{s\downarrow 0}\left[\dfrac{4}{\C}\dfrac{QC_b}{2^{s\overline{p}}p^+}\int_\mathbb{G}\frac{\varphi(|u(x)|)}{\|x\|_\mathbb{G}^{s\overline{p}}}\,dx-\frac{2QC_b}{2^{s\tilde{p}}\tilde{p}}(3^{s\tilde{p}}-1)\int_{\mathbb{G}}\frac{\varphi(|u(x)|)}{\|x\|_\mathbb{G}^{s\tilde{p}}}\,dx\right]\\
    &\geq\liminf_{s\downarrow 0}\dfrac{4}{\C}\dfrac{QC_b}{2^{s\overline{p}}p^+}\int_\mathbb{G}\frac{\varphi(|u(x)|)}{\|x\|_\mathbb{G}^{s\overline{p}}}\,dx-\limsup_{s\downarrow 0}\frac{2QC_b}{2^{s\tilde{p}}\tilde{p}}(3^{s\tilde{p}}-1)\int_{\mathbb{G}}\frac{\varphi(|u(x)|)}{\|x\|_\mathbb{G}^{s\tilde{p}}}\,dx\\
    &\geq \dfrac{4}{\C}\dfrac{QC_b}{p^+}\Phi_\varphi(u).
\end{align*}
\end{proof}
\begin{cor}
If $\liminf_{s\downarrow 0}s\Phi_{s,\varphi}(u)<\infty$, then $u\in L^\varphi(\mathbb{G})$.
\end{cor}
\begin{theorem}[Limsup estimate]\label{thlimsup}
For any $u\in\bigcup_{s\in (0,1)}W^{s,\varphi}(\mathbb{G})$ it holds that
\begin{equation}\label{limsup}
    \limsup_{s\downarrow 0}s\Phi_{s,\varphi}(u)\leq\C\dfrac{QC_b}{p^-}\Phi_\varphi(u).
\end{equation}
\end{theorem}
\begin{proof}
Let us first notice that, for any $\alpha>0$, in light of the Fubini's Theorem and by a change of variables, it holds that
\begin{align}\label{fubini}
\begin{split}
    \int_{\mathbb{G}}&\left(\int_{\{\|x\|_\mathbb{G}\leq\alpha\|y\|_\mathbb{G}\}}\varphi\left(\frac{|u(x)-u(y)|}{\|y^{-1}\cdot x\|_\mathbb{G}^s}\right)\,\frac{dx}{\|y^{-1}\cdot x\|_\mathbb{G}^Q}\right)\,dy\\
    &=\int_{\mathbb{G}}\left(\int_{\{\|y\|_\mathbb{G}\geq\frac{1}{\alpha}\|x\|_\mathbb{G}\}}\varphi\left(\frac{|u(x)-u(y)|}{\|y^{-1}\cdot x\|_\mathbb{G}^s}\right)\,\frac{dy}{\|y^{-1}\cdot y\|_\mathbb{G}^Q}\right)\,dx\\
    &=\int_{\mathbb{G}}\left(\int_{\{\|x\|_\mathbb{G}\geq\frac{1}{\alpha}\|y\|_\mathbb{G}\}}\varphi\left(\frac{|u(x)-u(y)|}{\|y^{-1}\cdot x\|_\mathbb{G}^s}\right)\,\frac{dx}{\|y^{-1}\cdot x\|_\mathbb{G}^Q}\right)\,dy.
\end{split}
\end{align}
An analogous relation holds when changing $\leq$ with $\geq$. Then 
\begin{align} \label{3}
\begin{split}
  &\Phi_{s,\varphi}(u)=\\ &=\int_{\mathbb{G}}\left(\int_{\{\frac{1}{2}\|y\|_\mathbb{G}<\|x\|_\mathbb{G}<\|y\|_\mathbb{G}\}} + \int_{\{\|x\|_\mathbb{G}\leq\frac{1}{2}\|y\|_\mathbb{G}\} } +  \int_{\{\|x\|_\mathbb{G}\geq\|y\|_\mathbb{G}\} } \right) \varphi\left(\frac{|u(x)-u(y)|}{\|y^{-1}\cdot x\|_\mathbb{G}^s}\right)\,\frac{dx\,dy}{\|y^{-1}\cdot x\|_\mathbb{G}^Q} \\
& =\int_{\mathbb{G}}\left(\int_{\{\frac{1}{2}\|y\|_\mathbb{G}<\|x\|_\mathbb{G}<\|y\|_\mathbb{G}\}} + \int_{\{\|x\|_\mathbb{G}\leq\frac{1}{2}\|y\|_\mathbb{G}\} } +  \int_{\{\|x\|_\mathbb{G}\leq\|y\|_\mathbb{G}\} } \right) \varphi\left(\frac{|u(x)-u(y)|}{\|y^{-1}\cdot x\|_\mathbb{G}^s}\right)\,\frac{dx\,dy}{\|y^{-1}\cdot x\|_\mathbb{G}^Q} \\  
&=2\int_{\mathbb{G}}\left(\int_{\{\frac{1}{2}\|y\|_\mathbb{G}<\|x\|_\mathbb{G}<\|y\|_\mathbb{G}\}} + \int_{\{\|x\|_\mathbb{G}\leq\frac{1}{2}\|y\|_\mathbb{G}\} }  \right) \varphi\left(\frac{|u(x)-u(y)|}{\|y^{-1}\cdot x\|_\mathbb{G}^s}\right)\,\frac{dx\,dy}{\|y^{-1}\cdot x\|_\mathbb{G}^Q}.
\end{split}
\end{align}
We divide the proof of the Theorem in two main steps:

\underline{\textit{Step 1:}} Let us show that
\begin{equation}\label{step1}
    \begin{split}
        2s\int_{\mathbb{G}}&\left(\int_{\{\|x\|_\mathbb{G}\leq\frac{1}{2}\|y\|_\mathbb{G}\}}\varphi\left(\frac{|u(x)-u(y)|}{\|y^{-1}\cdot x\|_\mathbb{G}^s}\right)\,\frac{dx}{\|y^{-1}\cdot x\|_\mathbb{G}^Q}\right)\,dy\\
        &\leq\C\dfrac{QC_b}{p^-}\int_\mathbb{G}\frac{\varphi(|u(x)|)}{\|x\|_\mathbb{G}^{s\tilde{p}}}\,dx+\C s2^{s\tilde{p}}C_b\int_\mathbb{G}\frac{\varphi(|u(x)|)}{\|x\|_\mathbb{G}^{s\tilde{p}}}\,dx.
    \end{split}
\end{equation}
\begin{proof}[Proof of Step 1]
First, by the triangular inequality, $\|y\|_\mathbb{G}\geq 2\|x\|_\mathbb{G}$ implies that
$$
\|y^{-1}\cdot x\|_\mathbb{G}\geq\|y\|_\mathbb{G}-\|x\|_\mathbb{G}\geq\|x\|_\mathbb{G} \qquad \text{and} \qquad \|y^{-1}\cdot x\|_\mathbb{G}\geq\|y\|_\mathbb{G}-\|x\|_\mathbb{G}\geq\frac{1}{2}\|y\|_\mathbb{G}
$$ 
that is, 
\begin{equation}\label{sets2}
    \{\|y\|_\mathbb{G}\geq2\|x\|_\mathbb{G}\}\subset\{\|y^{-1}\cdot x\|_\mathbb{G}\geq\|x\|_\mathbb{G}\},
\end{equation}
\begin{equation}\label{sets3}
    \|y^{-1}\cdot x\|_\mathbb{G}\geq\frac{1}{2}\|y\|_\mathbb{G}.
\end{equation}
Then, by \eqref{sets2}, the Fubini's Theorem and by Proposition \ref{prop.radialfunct}
\begin{align} \label{4}
\begin{split}
    \int_{\mathbb{G}}\bigg(\int_{\{\|x\|_\mathbb{G}\leq\frac{1}{2}\|y\|_\mathbb{G}\}}&\frac{\varphi(|u(x)|)}{\|y^{-1}\cdot x\|_\mathbb{G}^{s\tilde{p}+Q}}\,dx\bigg)\,dy\\
    &=\int_{\mathbb{G}}\left(\int_{\{\|y\|_\mathbb{G}\geq2\|x\|_\mathbb{G}\}}\frac{dy}{\|y^{-1}\cdot x\|_\mathbb{G}^{s\tilde{p}+Q}}\right)\varphi(|u(x)|)\,dx\\
    &\leq\int_{\mathbb{G}}\left(\int_{\{\|y^{-1}\cdot x\|_\mathbb{G}\geq\|x\|_\mathbb{G}\}}\frac{dy}{\|y^{-1}\cdot x\|_\mathbb{G}^{s\tilde{p}+Q}}\right)\varphi(|u(x)|)\,dx\\
    &=QC_b\int_\mathbb{G}\left(\int_{\|x\|_\mathbb{G}}^{+\infty}r^{-s\tilde{p}-1}\,dr\right)\varphi(|u(x)|)\,dx\\
    &=\dfrac{QC_b}{s\tilde{p}}\int_\mathbb{G}\frac{\varphi(|u(x)|)}{\|x\|_\mathbb{G}^{s\tilde{p}}}\,dx\leq\dfrac{QC_b}{sp^-}\int_\mathbb{G}\frac{\varphi(|u(x)|)}{\|x\|_\mathbb{G}^{s\tilde{p}}}\,dx.
\end{split}
\end{align}
Moreover, by \eqref{sets3}, we get
\begin{align}\label{5}
\begin{split} 
    \int_{\mathbb{G}}\bigg( \int_{\{\|x\|_\mathbb{G}\leq\frac{1}{2}\|y\|_\mathbb{G}\}}&\frac{ \varphi(|u(y)|)}{\|y^{-1}\cdot x\|_\mathbb{G}^{s\tilde{p}+Q}}\,dx\bigg) \,dy\\
     &\leq 2^{s\tilde{p}+Q}\int_{\mathbb{G}}\left(\int_{\{\|x\|_\mathbb{G}\leq\frac{1}{2}\|y\|_\mathbb{G}\}}\,dx\right)\frac{\varphi(|u(y)|)}{\|y\|_\mathbb{G}^{s\tilde{p}+Q}}\,dy\\
    &=2^{s\tilde{p}+Q}QC_b\int_{\mathbb{G}}\left(\int_0^{\frac{1}{2}\|y\|_\mathbb{G}}r^{Q-1}\,dr\right)\frac{\varphi(|u(y)|)}{\|y\|_\mathbb{G}^{s\tilde{p}+Q}}\,dy\\
    &=2^{s\tilde{p}}C_b\int_{\mathbb{G}}\frac{\varphi(|u(y)|)}{\|y\|_\mathbb{G}^{s\tilde{p}}}\,dy.
\end{split}
\end{align}
Then, by \eqref{G1}, \eqref{G2}, \eqref{4} and \eqref{5}, we finally have
\begin{align*}
    \int_{\mathbb{G}}&\left(\int_{\{\|x\|_\mathbb{G}\leq\frac{1}{2}\|y\|_\mathbb{G}\}}\varphi\left(\frac{|u(x)-u(y)|}{\|y^{-1}\cdot x\|_\mathbb{G}^s}\right)\,\frac{dx}{\|y^{-1}\cdot x\|_\mathbb{G}^Q}\right)\,dy\\
    &\leq\frac{\C}2  \int_{\mathbb{G}}\left(\int_{\{\|x\|_\mathbb{G}\leq\frac{1}{2}\|y\|_\mathbb{G}\}} \left[ \varphi\left(\frac{|u(x)|}{\|y^{-1}\cdot x\|_\mathbb{G}^s}\right) + \varphi\left(\frac{|u(y)|}{\|y^{-1}\cdot x\|_\mathbb{G}^s}\right) \right]\,\frac{dx}{\|y^{-1}\cdot x\|_\mathbb{G}^Q}\right)\,dy\\
    &\leq\frac{\C}{2}\dfrac{QC_b}{sp^-}\int_\mathbb{G}\frac{\varphi(|u(x)|)}{\|x\|_\mathbb{G}^{s\tilde{p}}}\,dx+ \C 2^{s\tilde{p}-1}C_b\int_{\mathbb{G}}\frac{\varphi(|u(x)|)}{\|x\|_\mathbb{G}^{s\tilde{p}}}\,dx
\end{align*}
and hence \eqref{step1} follows.
\end{proof}
\underline{\textit{Step 2:}} Let us show that
\begin{equation}\label{step2}
    \limsup_{s\downarrow 0}s\int_{\mathbb{G}}\left(\int_{\{\frac{1}{2}\|y\|_\mathbb{G}<\|x\|_\mathbb{G}<\|y\|_\mathbb{G}\}}\varphi\left(\frac{|u(x)-u(y)|}{\|y^{-1}\cdot x\|_\mathbb{G}^s}\right)\,\frac{dx}{\|y^{-1}\cdot x\|_\mathbb{G}^Q}\right)\,dy=0.
\end{equation}
\begin{proof}[Proof of Step 2]
Fixed $N>1$ we write
\begin{equation} \label{ee0}
\int_{\mathbb{G}}\left(\int_{\{\frac{1}{2}\|y\|_\mathbb{G}<\|x\|_\mathbb{G}<\|y\|_\mathbb{G}\}}\varphi\left(\frac{|u(x)-u(y)|}{\|y^{-1}\cdot x\|_\mathbb{G}^s}\right)\,\frac{dx}{\|y^{-1}\cdot x\|_\mathbb{G}^Q}\right)\,dy = (i)+(ii), 
\end{equation}
where
\begin{align*}
(i):=\int_{\mathbb{G}}\left(\int_{\{\frac{1}{2}\|y\|_\mathbb{G}<\|x\|_\mathbb{G}<\|y\|_\mathbb{G},\|y^{-1}\cdot x\|_\mathbb{G}\leq N\}}\varphi\left(\frac{|u(x)-u(y)|}{\|y^{-1}\cdot x\|_\mathbb{G}^s}\right)\,\frac{dx}{\|y^{-1}\cdot x\|_\mathbb{G}^Q}\right)\,dy\\
(ii):=\int_{\mathbb{G}}\left(\int_{\{\frac{1}{2}\|y\|_\mathbb{G}<\|x\|_\mathbb{G}<\|y\|_\mathbb{G},\|y^{-1}\cdot x\|_\mathbb{G}> N\}}\varphi\left(\frac{|u(x)-u(y)|}{\|y^{-1}\cdot x\|_\mathbb{G}^s}\right)\,\frac{dx}{\|y^{-1}\cdot x\|_\mathbb{G}^Q}\right)\,dy.
\end{align*}
Observe that fixed $\overline{s}\in(0,1)$ such that $\overline{s}>s$, by \eqref{G1}, $(i)$ can be bounded as 
\begin{align} \label{6}
 N^{(\overline{s}-s)\tilde{p}}\int_{\mathbb{G}}\left(\int_{\{\frac{1}{2}\|y\|_\mathbb{G}<\|x\|_\mathbb{G}<\|y\|_\mathbb{G},\|y^{-1}\cdot x\|_\mathbb{G}\leq N\}}\varphi\left(\frac{|u(x)-u(y)|}{\|y^{-1}\cdot x\|_\mathbb{G}^{\overline{s}}}\right)\,\frac{dx}{\|y^{-1}\cdot x\|_\mathbb{G}^Q}\right)\,dy.
\end{align}
Moreover, by \eqref{G1} and \eqref{G2}, we get
\begin{align*}
    (ii)
    &\leq\frac{\C}{2}\int_{\mathbb{G}}\left(\int_{\{\frac{1}{2}\|y\|_\mathbb{G}<\|x\|_\mathbb{G}<\|y\|_\mathbb{G},\|y^{-1}\cdot x\|_\mathbb{G}>N\}} \frac{ \varphi(|u(x)|)}{\|y^{-1}\cdot x\|_\mathbb{G}^{sp^-+Q}} \,dx\right)\,dy\\
    &\quad+\frac{\C}{2}\int_{\mathbb{G}}\left(\int_{\{\frac{1}{2}\|y\|_\mathbb{G}<\|x\|_\mathbb{G}<\|y\|_\mathbb{G},\|y^{-1}\cdot x\|_\mathbb{G}>N\}}\,\frac{dx}{\|y^{-1}\cdot x\|_\mathbb{G}^{sp^-+Q}}\right)\varphi(|u(y)|)\,dy\\
    &:=\frac{\C}{2}\left((i')+(ii')\right).
\end{align*}
Let us note that, by the Fubini's Theorem and a change of variables
\begin{align*}
    (i')&=\int_{\mathbb{G}}\left(\int_{\{\|x\|_\mathbb{G}<\|y\|_\mathbb{G}<2\|x\|_\mathbb{G},\|y^{-1}\cdot x\|_\mathbb{G}>N\}}\,\frac{dy}{\|y^{-1}\cdot x\|_\mathbb{G}^{sp^-+Q}}\right)\varphi(|u(x)|)\,dx\\
    &=\int_{\mathbb{G}}\left(\int_{\{\|y\|_\mathbb{G}<\|x\|_\mathbb{G}<2\|y\|_\mathbb{G},\|y^{-1}\cdot x\|_\mathbb{G}>N\}}\,\frac{dx}{\|y^{-1}\cdot x\|_\mathbb{G}^{sp^-+Q}}\right)\varphi(|u(y)|)\,dy\\
    &\leq (ii')
\end{align*}
where we have used the inclusion
\begin{align*}
    &\{\|y\|_\mathbb{G}<\|x\|_\mathbb{G}<2\|y\|_\mathbb{G},\|y^{-1}\cdot x\|_\mathbb{G}>N\}
    \subset\left\{\frac{1}{2}\|y\|_\mathbb{G}<\|x\|_\mathbb{G}<\|y\|_\mathbb{G},\|y^{-1}\cdot x\|_\mathbb{G}>N\right\}.
\end{align*}
The last two inequalities lead to
\begin{equation}\label{7}
(ii)\leq \C (ii').
\end{equation}
Now, since
\begin{align*}
    \left\{ \frac{1}{2}\|y\|_\mathbb{G}<\|x\|_\mathbb{G}<\|y\|_\mathbb{G},\|y^{-1}\cdot x\|_\mathbb{G}>N\right\}
    \subset\left\{\|y\|_\mathbb{G}>\frac{N}{2},\|y^{-1}\cdot x\|_\mathbb{G}>N\right\},
\end{align*}
using Proposition \ref{prop.radialfunct}, we get
\begin{align}\label{8}
\begin{split}
    (ii')    &\leq  \int_{{\{\|y\|_\mathbb{G}>\frac{N}{2}}\}}\left(\int_{\{\|y^{-1}\cdot x\|_\mathbb{G}>N\}}\,\frac{dx}{\|y^{-1}\cdot x\|_\mathbb{G}^{sp^-+Q}}\right)\varphi(|u(y)|)\,dy\\
    &=  QC_b\int_{{\{\|y\|_\mathbb{G}>\frac{N}{2}}\}}\left(\int_N^{+\infty}r^{-sp^--1}\,dr\right)\varphi(|u(y)|)\,dy\\
    &=  \frac{QC_b}{sp^-N^{sp^-}}\int_{{\{\|y\|_\mathbb{G}>\frac{N}{2}}\}}\varphi(|u(y)|)\,dy.
\end{split}
\end{align}

Then, from \eqref{ee0}, \eqref{6}, \eqref{7} and \eqref{8}
\begin{align*}
    \limsup_{s\downarrow 0}s\int_{\mathbb{G}}&\left(\int_{\{\frac{1}{2}\|y\|_\mathbb{G}<\|x\|_\mathbb{G}<\|y\|_\mathbb{G}\}}\varphi\left(\frac{|u(x)-u(y)|}{\|y^{-1}\cdot x\|_\mathbb{G}^s}\right)\,\frac{dx}{\|y^{-1}\cdot x\|_\mathbb{G}^Q}\right)\,dy\leq\\ 
    &\leq \limsup_{s\downarrow 0}\frac{\C QC_b}{p^-N^{sp^-}}\int_{{\{\|y\|_\mathbb{G}>\frac{N}{2}}\}}\varphi(|u(y)|)\,dy=0
\end{align*}
for $N$ sufficiently large.
\end{proof}
Finally, we obtain the limsup inequality \eqref{limsup} by gathering  \eqref{3},  \underline{\textit{Step 1}}, \underline{\textit{Step 2}} and by using the Fatou's Lemma.
\end{proof}
As a corollary of Theorem \ref{thliminf} and Theorem \ref{thlimsup}, we get Theorem \ref{MS1}.
\begin{rmk}\label{rmkmin}
We remind that, as mentioned in the Introduction, Theorem \ref{MS1} does not provide a complete generalization of \eqref{MSclassic}, even in the prototype case $\varphi(t)=t^p$, i.e. $p^+=p^-$, being
$\frac{4}{\C}<\C$
by construction (see Definition \ref{delta2} for details).
\end{rmk}
\medskip

Let us conclude the paper with the proof of Theorem \ref{MS2}.
\begin{proof}[Proof of Theorem \ref{MS2}]
The first part of the proof, that is the $\liminf$-inequality, follows from similar arguments of Theorem \ref{thliminf}. Indeed, by \eqref{G1}, the Minkowski inequality,  \eqref{0} and \eqref{000}, we have
\begin{align*}
    &s\int_{\mathbb{G}}\left(\int_{\{\|y^{-1}\cdot x\|_\mathbb{G}>2\|x\|_\mathbb{G}\}}\varphi\left(\frac{|u(x)|}{\|y^{-1}\cdot x\|_\mathbb{G}^s}\right)\,\frac{dy}{\|y^{-1}\cdot x\|_\mathbb{G}^Q}\right)\,dx\\
    &\leq
    (\varphi\circ\varphi^{-1})\left[s\int_{\mathbb{G}}\left(\int_{\{\|y^{-1}\cdot x\|_\mathbb{G}>2\|x\|_\mathbb{G}\}}\varphi\left(\frac{|u(x)-u(y)|}{\|y^{-1}\cdot x\|_\mathbb{G}^s}+\frac{|u(y)|}{\|y^{-1}\cdot x\|_\mathbb{G}^s}\right)\,\frac{dy}{\|y^{-1}\cdot x\|_\mathbb{G}^Q}\right)\,dx\right]\\  
    &\leq
    \varphi\Bigg[\varphi^{-1}\left(s\int_{\mathbb{G}}\left(\int_{\{\|y^{-1}\cdot x\|_\mathbb{G}>2\|x\|_\mathbb{G}\}}\varphi\left(\frac{|u(x)-u(y)|}{\|y^{-1}\cdot x\|_\mathbb{G}^s}\right)\,\frac{dy}{\|y^{-1}\cdot x\|_\mathbb{G}^Q}\right)\,dx\right)\\
    &\quad+\varphi^{-1}\left(s\int_{\mathbb{G}}\left(\int_{\{\|y^{-1}\cdot x\|_\mathbb{G}>2\|x\|_\mathbb{G}\}}\varphi\left(\frac{|u(y)|}{\|y^{-1}\cdot x\|_\mathbb{G}^s}\right)\,\frac{dy}{\|y^{-1}\cdot x\|_\mathbb{G}^Q}\right)\,dx\right)\Bigg]\\
    &\leq\varphi\left[\varphi^{-1}\left(\frac{1}{2}s\Phi_{s,\varphi}(u)\right)+\varphi^{-1}\left(\frac{QC_b}{2^{s\tilde{p}}\tilde{p}}(3^{s\tilde{p}}-1)\int_{\mathbb{G}}\frac{\varphi(|u(y)|)}{\|y\|_\mathbb{G}^{s\tilde{p}}}\,dy\right)\right],
\end{align*}
and then, by \eqref{1} and the monotonicity of $\varphi^{-1}$, we get
\begin{align*}
    \varphi^{-1}\left(\frac{1}{2}s\Phi_{s,\varphi}(u)\right)&+\varphi^{-1}\left(\frac{QC_b}{2^{s\tilde{p}}\tilde{p}}(3^{s\tilde{p}}-1)\int_{\mathbb{G}}\frac{\varphi(|u(x)|)}{\|x\|_\mathbb{G}^{s\tilde{p}}}\,dx\right)\\
    &\geq\varphi^{-1}\left(\dfrac{QC_b}{2^{s\overline{p}}p^+}\int_\mathbb{G}\frac{\varphi(|u(x)|)}{\|x\|_\mathbb{G}^{s\overline{p}}}\,dx\right).
\end{align*}
Then, by the properties of the $\liminf$, we have
\begin{align*}
    \liminf_{s\downarrow0}\varphi^{-1}\left(\frac{1}{2}s\Phi_{s,\varphi}(u)\right)&\geq\liminf_{s\downarrow0}\left(\varphi^{-1}\left(\dfrac{QC_b}{2^{s\overline{p}}p^+}\int_\mathbb{G}\frac{\varphi(|u(x)|)}{\|x\|_\mathbb{G}^{s\overline{p}}}\,dx\right)\right)\\
    &\quad+\liminf_{s\downarrow0}\left(-\varphi^{-1}\left(\frac{QC_b}{2^{s\tilde{p}}\tilde{p}}(3^{s\tilde{p}}-1)\int_{\mathbb{G}}\frac{\varphi(|u(x)|)}{\|x\|_\mathbb{G}^{s\tilde{p}}}\,dx\right)\right)
\end{align*}
the right hand side of the inequality above is equal to
$$
\liminf_{s\downarrow0}\left(\varphi^{-1}\left(\dfrac{QC_b}{2^{s\overline{p}}p^+}\int_\mathbb{G}\frac{\varphi(|u(x)|)}{\|x\|_\mathbb{G}^{s\overline{p}}}\,dx\right)\right)-\limsup_{s\downarrow0}\left(\varphi^{-1}\left(\frac{QC_b}{2^{s\tilde{p}}\tilde{p}}(3^{s\tilde{p}}-1)\int_{\mathbb{G}}\frac{\varphi(|u(x)|)}{\|x\|_\mathbb{G}^{s\tilde{p}}}\,dx\right)\right)
$$
and, by the continuity of $\varphi^{-1}$, it is equivalent to
$$
\varphi^{-1}\left(\liminf_{s\downarrow0}\left(\dfrac{QC_b}{2^{s\overline{p}}p^+}\int_\mathbb{G}\frac{\varphi(|u(x)|)}{\|x\|_\mathbb{G}^{s\overline{p}}}\,dx\right)\right)-\varphi^{-1}\left(\limsup_{s\downarrow0}\left(\frac{QC_b}{2^{s\tilde{p}}\tilde{p}}(3^{s\tilde{p}}-1)\int_{\mathbb{G}}\frac{\varphi(|u(x)|)}{\|x\|_\mathbb{G}^{s\tilde{p}}}\,dx\right)\right).
$$
Therefore, by the Fatou's Lemma, we obtain that
$$
\liminf_{s\downarrow0}\varphi^{-1}\left(\frac{1}{2}s\Phi_{s,\varphi}(u)\right)\geq\varphi^{-1}\left(\dfrac{QC_b}{p^+}\Phi_\varphi(u)\right)-\varphi^{-1}(0).
$$
Thus, by the monotonicity and continuity of $\varphi$, and being $\varphi^{-1}(0)=0$, we get that
\begin{align*}
    \liminf_{s\downarrow0}\left(\frac{1}{2}s\Phi_{s,\varphi}(u)\right)&=\varphi\left(\varphi^{-1}\liminf_{s\downarrow0}\left(\frac{1}{2}s\Phi_{s,\varphi}(u)\right)\right)\\
    &=\varphi\left(\liminf_{s\downarrow0}\varphi^{-1}\left(\frac{1}{2}s\Phi_{s,\varphi}(u)\right)\right)\\
    &\geq\varphi\left(\varphi^{-1}\left(\dfrac{QC_b}{p^+}\Phi_\varphi(u)\right)\right)=\dfrac{QC_b}{p^+}\Phi_\varphi(u),
\end{align*}
as desired.

In order to conclude the proof of the Theorem, let us refine the proof of the $\limsup$ estimate. To this aim, it is not necessary to modify \underline{\textit{Step 2}} of Theorem \ref{thlimsup}. Instead, the following ``new version" of \underline{\textit{Step 1}} is needed:

\underline{\textit{Step 1b}}:
\begin{equation}\label{step1b}
    \begin{split}
        2s\int_{\mathbb{G}}&\left(\int_{\{\|x\|_\mathbb{G}\leq\frac{1}{2}\|y\|_\mathbb{G}\}}\varphi\left(\frac{|u(x)-u(y)|}{\|y^{-1}\cdot x\|_\mathbb{G}^s}\right)\,\frac{dx}{\|y^{-1}\cdot x\|_\mathbb{G}^Q}\right)\,dy\\
        &\leq\varphi\left[\varphi^{-1}\left(2\dfrac{QC_b}{p^-}\int_\mathbb{G}\frac{\varphi(|u(x)|)}{\|x\|_\mathbb{G}^{s\tilde{p}}}\,dx\right)+\varphi^{-1}\left(2 s2^{s\tilde{p}}C_b\int_\mathbb{G}\frac{\varphi(|u(x)|)}{\|x\|_\mathbb{G}^{s\tilde{p}}}\,dx\right)\right].
    \end{split}
\end{equation}
Indeed, by \eqref{G1} and the Minkowski inequality, we have that
\begin{align*}
    2s\int_{\mathbb{G}}&\left(\int_{\{\|x\|_\mathbb{G}\leq\frac{1}{2}\|y\|_\mathbb{G}\}}\varphi\left(\frac{|u(x)-u(y)|}{\|y^{-1}\cdot x\|_\mathbb{G}^s}\right)\,\frac{dx}{\|y^{-1}\cdot x\|_\mathbb{G}^Q}\right)\,dy\\
    &\leq\varphi\Bigg[\varphi^{-1}\left(2s\int_{\mathbb{G}}\left(\int_{\{\|x\|_\mathbb{G}\leq\frac{1}{2}\|y\|_\mathbb{G}\}}\varphi\left(\frac{|u(x)|}{\|y^{-1}\cdot x\|_\mathbb{G}^s}\right)\,\frac{dx}{\|y^{-1}\cdot x\|_\mathbb{G}^Q}\right)\,dy\right)\\
    &\quad+\varphi^{-1}\left(2s\int_{\mathbb{G}}\left(\int_{\{\|x\|_\mathbb{G}\leq\frac{1}{2}\|y\|_\mathbb{G}\}}\varphi\left(\frac{|u(y)|}{\|y^{-1}\cdot x\|_\mathbb{G}^s}\right)\,\frac{dx}{\|y^{-1}\cdot x\|_\mathbb{G}^Q}\right)\,dy\right)\Bigg]
\end{align*}
and hence,  \eqref{4} and \eqref{5} lead to  \eqref{step1b}.

Thus, named
$$
(i):=\int_{\mathbb{G}}\left(\int_{\{\frac{1}{2}\|y\|_\mathbb{G}<\|x\|_\mathbb{G}<\|y\|_\mathbb{G}\}}\varphi\left(\frac{|u(x)-u(y)|}{\|y^{-1}\cdot x\|_\mathbb{G}^s}\right)\,\frac{dx}{\|y^{-1}\cdot x\|_\mathbb{G}^Q}\right)\,dy
$$
and
$$
(ii):=\int_{\mathbb{G}}\left(\int_{\{\|x\|_\mathbb{G}\leq\frac{1}{2}\|y\|_\mathbb{G}\}}\varphi\left(\frac{|u(x)-u(y)|}{\|y^{-1}\cdot x\|_\mathbb{G}^s}\right)\,\frac{dx}{\|y^{-1}\cdot x\|_\mathbb{G}^Q}\right)\,dy
$$
by \eqref{3}, \underline{\textit{Step 1b}}, \underline{\textit{Step 2}}, the continuity of $\varphi$ and $\varphi^{-1}$ and the Fatou's Lemma, we finally get
\begin{align*}
    &\limsup_{s\downarrow 0}s\Phi_{s,\varphi}(u) \leq\limsup_{s\downarrow 0}2s \,(i)+\limsup_{s\downarrow 0}2s\, (ii)\\
    &\quad\leq
    \limsup_{s\downarrow 0}\varphi\Bigg[\varphi^{-1}\left(2\dfrac{QC_b}{p^-}\int_\mathbb{G}\frac{\varphi(|u(x)|)}{\|x\|_\mathbb{G}^{s\tilde{p}}}\,dx\right)+\varphi^{-1}\left(2s2^{s\tilde{p}}C_b\int_{\mathbb{G}}\frac{\varphi(|u(x)|)}{\|x\|_\mathbb{G}^{s\tilde{p}}}\,dx\right)\Bigg]\\
    &\quad\leq
    \varphi\Bigg[\varphi^{-1}\limsup_{s\downarrow 0}\left(2\dfrac{QC_b}{p^-}\int_\mathbb{G}\frac{\varphi(|u(x)|)}{\|x\|_\mathbb{G}^{s\tilde{p}}}\,dx\right)+\varphi^{-1}\left(\limsup_{s\downarrow 0}2s2^{s\tilde{p}}C_b\int_{\mathbb{G}}\frac{\varphi(|u(x)|)}{\|x\|_\mathbb{G}^{s\tilde{p}}}\,dx\right)\Bigg]\\
    &\quad\leq\varphi\Bigg[\varphi^{-1}\left(2\dfrac{QC_b}{p^-}\Phi_\varphi(u)\right)+\varphi^{-1}(0)\Bigg]=2\dfrac{QC_b}{p^-}\Phi_\varphi(u)
\end{align*}
concluding the proof.
\end{proof}


\begin{thebibliography}{10}
\bibitem{AmbDepMart}
{\sc L. Ambrosio, G. De Philippis and L. Martinazzi},
{\em $\Gamma-$convergence of nonlocal perimeter functionals,}
Manuscripta Math. {\bf 134}, (2011), 377--403.

\bibitem{B} 
{\sc D. Barbieri}, 
{\em Approximations of Sobolev norms in Carnot groups}, 
Commun. Contemp. Math. {\bf 13}(5), (2011), 765–-794.

\bibitem{BBM}
{\sc J. Bourgain, H. Brezis and P. Mironescu},
{\em Another look at Sobolev spaces}, 
in \emph{Optimal Control and Partial Differential Equations. A Volume in Honor of Professor Alain Bensoussan's 60th Birthday},
(eds. J. L. Menaldi, E. Rofman and A. Sulem), IOS Press, Amsterdam, (2001), 439--455.


\bibitem{bre}
{\sc H. Brezis},
{\it How to recognize constant functions. Connections with Sobolev spaces},
Russian Mathematical Surveys \textbf{57}, (2002), 693--708.

\bibitem{bre-linc}
{\sc H. Brezis}, 
{\em New approximations of the total variation and filters in imaging}, 
Rend Accad. Lincei {\bf 26}, (2015), 223--240.

\bibitem{BHN}
{\sc H. Brezis and H.-M. Nguyen},
{\em Non-local functionals related to the total variation and connections with image processing}, Annals of PDE. Journal Dedicated to the Analysis of Problems from Physical Sciences, {\bf 4} (1), (2018), Art. 9, 77.

\bibitem{BHN2}
{\sc H. Brezis and H.-M. Nguyen},
{\em The BBM formula revisited}, 
Atti Accad. Naz. Lincei Rend. Lincei Mat. Appl. {\bf 27}, (2016), 515--533.

\bibitem{BHN3}
{\sc H. Brezis and H.-M. Nguyen},
{\em Two subtle convex nonlocal approximations of the BV-norm},
Nonlinear Anal. {\bf 137}, (2016), 222–-245. 

\bibitem{BLU}
{\sc A. Bonfiglioli, E. Lanconelli and F. Uguzzoni}, 
{\em Stratified Lie Groups and Potential Theory for Their Sub-Laplacians}, 
Springer, 2007.

\bibitem{BS}
{\sc J. Fern\'{a}ndez Bonder and A.M. Salort},
{\em Fractional order Orlicz-Sobolev spaces}, 
J. Funct. Anal. {\bf 277} (2), (2019), 333–-367.

\bibitem{BS2}
{\sc J. Fern\'{a}ndez Bonder and A.M. Salort},
{\em Magnetic Fractional order Orlicz-Sobolev spaces}, preprint.
https://arxiv.org/abs/1812.05998.

\bibitem{CV}
{\sc L. Caffarelli and E. Valdinoci}, 
{\em Regularity properties of nonlocal minimal surfaces via limiting arguments}, 
Adv. Math. {\bf 248}, (2013), 843--871. 

\bibitem{Cui}
{\sc X. Cui, N. Lam and G. Lu}, 
{\em New characterizations of Sobolev spaces in the Heisenberg group},
J. Funct. Anal. \textbf{267} (2014), 2962--2994.

\bibitem{Cygan}
{\sc J. Cygan}, 
{\em Subadditivity of homogeneous norms on certain nilpotent Lie groups}, 
Proc. Amer. Math. Soc. {\bf 83}, (1981), 69--70.

\bibitem{davila}
{\sc J. Davila}, 
{\em On an open question about functions of bounded variation}, 
Calc. Var. Partial Differential Equations {\bf 15}, (2002), 519--527. 

\bibitem{DHHR} 
{\sc L. Diening, P. Harjulehto, P. H{\"a}st{\"o} and M. Ru\v{z}i\v{c}ka}, 
{\em Lebesgue and Sobolev spaces with variable exponents}. 
Lecture Notes in Mathematics, 2017. Springer, Heidelberg, 2011.

\bibitem{DiFiPaVa}
{\sc S. Dipierro, A. Figalli, G. Palatucci and E. Valdinoci}, 
{\em Asymptotics of the s-perimeter as $s \to 0$}, 
Discrete Contin. Dyn. Syst. {\bf 33}, (2013), 2777--2790.

\bibitem{FMPPS}
{\sc F. Ferrari, M. Miranda, D. Pallara, A. Pinamonti and Y. Sire}
{\em Fractional Laplacians, perimeters and heat semigroups in Carnot groups}, 
Discrete Contin. Dyn. Syst. (Series S), {\bf 11}(3), (2018), 477--491.

\bibitem{FF}
{\sc F. Ferrari and B. Franchi},
{\em Harnack inequality for fractional sub-{L}aplacians in {C}arnot groups}. 
Mathematische Zeitschrift, {\bf 279} (1-2), (2015), 435--458.

\bibitem{FS} 
{\sc G.B. Folland and E.M. Stein}, 
{\em Hardy spaces on homogeneous groups}. 
Mathematical Notes, 28. Princeton University Press, N.J.; University of Tokyo Press, Tokyo, 1982.

\bibitem{FSSC} 
{\sc B. Franchi, R.P. Serapioni and F. Serra Cassano}, 
{\em Rectifiability and perimeter in the Heisenberg Group}, 
Math. Ann., {\bf 321} (3), (2001), 479--531.

\bibitem{FSSC2} 
{\sc B. Franchi, R.P. Serapioni and F. Serra Cassano}, 
{\em On the structure of finite perimeter sets in step 2 {C}arnot groups}, 
J. Geom. Anal. {\bf 13} (3), (2003), 421--466.

\bibitem{HHK} 
{\sc P. Harjulehto, P. H{\"a}st{\"o} and R. Kl{\'e}n}, 
{\em Basic properties of generalized Orlicz spaces}, 
Citeseer (2015).

\bibitem{KR}
{\sc M.A. Krasnosel'ski\u{i} and Ya. B. Ruticki\u{i}}, 
\emph{Convex Functions and Orlicz Spaces}, 
Translated from the first russian edition by Leo F. Boron, P. Noordhoff Ltd. - Groningen, 1961.

\bibitem{KrMo}
{\sc A. Kreum and O. Mordhorst},
{\em Fractional Sobolev norms and BV functions on manifolds},
Nonlinear Analysis {\bf 187}, (2019), 450--466.

\bibitem{Spe}
{\sc G. Leoni and D. Spector}, 
\emph{Characterization of Sobolev and BV spaces},
J. Funct. Anal. {\bf 261} (2011), 2926--2958.

\bibitem{Spe2}
{\sc G. Leoni and D. Spector}, 
{\it Corrigendum to ''Characterization of Sobolev and BV spaces''}, 
J. Funct. Anal. {\bf 266} (2014), 1106--1114.

\bibitem{Ludwig1}
{\sc M.\ Ludwig}, 
{\em Anisotropic fractional Sobolev norms}, 
Adv. Math. \textbf{252} (2014), 150--157.

\bibitem{Ludwig2}
{\sc M.\ Ludwig}, 
{\em Anisotropic fractional perimeters},
J. Differential Geom. \textbf{96} (2014), 77--93. 
\bibitem{L} 
{\sc W.A.J. Luxemburg}, 
{\em Banach function spaces}, 
Thesis, Technische Hogeschool TU Delft 1955.

\bibitem{MaPi}
{\sc A. Maalaoui and A. Pinamonti},
{\em Interpolations and fractional Sobolev spaces in Carnot groups},
Nonlinear Analysis {\bf 179}, (2019), 91--104.

\bibitem{MS} 
{\sc V. Maz'ya and T. Shaposhnikova}, 
{\em On the {B}ourgain, {B}rezis, and {M}ironescu theorem concerning limiting embeddings of fractional {S}obolev spaces}, 
J. Funct. Anal. {\bf 195} (2), (2002), 230--238.

\bibitem{Ng06}
{\sc H.-M. Nguyen}, 
{\em Some new characterizations of Sobolev spaces},
J. Funct. Anal. \textbf{237}, (2006), 689--720.

\bibitem{Ng08}
{\sc H.-M. Nguyen},  
{\em Further characterizations of Sobolev spaces}, 
J. Eur.  Math. Soc. (JEMS) \textbf{10}, (2008), 191--229.

\bibitem{Ng11}
{\sc H.-M. Nguyen},
{\em $\Gamma$-convergence, Sobolev norms, and BV functions}, 
Duke Math. J. \textbf{157}, (2011), 495--533.
  
\bibitem{Ng11bis}
{\sc H.-M. Nguyen},
{\em Some inequalities related to Sobolev norms},
Calc. Var. Partial Differential Equations {\bf 41}, (2011), 483--509.


\bibitem{NPSV}
{\sc H.-M. Nguyen, A. Pinamonti, M. Squassina and E. Vecchi},
{\em New characterizations of magnetic Sobolev spaces}, 
Advances in Nonlinear Analysis {\bf 7} (2), (2018), 227--245.

\bibitem{NPSV2}
{\sc H.-M. Nguyen, A. Pinamonti, M. Squassina and E. Vecchi},
{\em Some characterizations of magnetic Sobolev spaces}, 
to appear in Complex Variables and Elliptic Equations. \\
doi: 10.1080/17476933.2018.1520850.

\bibitem{P} 
{\sc P. Pansu}, 
{\em M\'{e}triques de {C}arnot-{C}arath\'{e}odory et quasi isom\'{e}tries des espaces sym\'{e}triques de rang un}, 
Ann. of Math. {\bf 129} (2), (1989), 1--60.

\bibitem{PKJF} 
{\sc L. Pick, A. Kufner, O. John and S. Fu\v{c}\'{\i}k}, 
{\em Function spaces. {V}ol. 1}, 
De Gruyter Series in Nonlinear Analysis and Applications {\bf 129} (14), (2013), xvi+479 pp.

\bibitem{PSV} 
{\sc A. Pinamonti, M. Squassina and E. Vecchi}, 
{\em The {M}az'ya-{S}haposhnikova limit in the magnetic setting}, 
J. Math. Anal. Appl. {\bf 449} (2), (2017), 1152--1159.

\bibitem{PSV2}
{\sc A. Pinamonti, M. Squassina and E. Vecchi},
{\em Magnetic BV-functions and the Bourgain-Brezis-Mironescu formula}, 
Adv. Calc. Var. {\bf 12} (3), (2019), 225--252.

\bibitem{Ponce}
{\sc A. Ponce},
{\em A new approach to Sobolev spaces and connections to $\Gamma$-convergence}, 
Calculus of Variations and Partial Differential Equations {\bf 19} (3), (2004), 229--255.

\bibitem{SqVo}
{\sc M. Squassina and B. Volzone}, 
{\em Bourgain-Brezis-Mironescu formula for magnetic operators},
C. R. Math. Acad. Sci. Paris {\bf 354} (2016), 825--831.
\end{thebibliography}
\end{document}